\documentclass[oneside]{amsart}

\newcommand{\docclass}{amsart}

\newcommand{\listitemhack}[1]{(#1)}	

	\usepackage[T1]{fontenc}
\usepackage[latin9]{inputenc}
\usepackage{amsthm}
\usepackage{amssymb}
\usepackage{amstext}
\usepackage{amsmath}

\usepackage{ifthen}
\usepackage{xstring}

\usepackage[numbers]{natbib}

\usepackage{verbatim}
\usepackage{wasysym}
\usepackage{xargs}[2008/03/08]

\usepackage{graphicx}

\usepackage{MnSymbol}

\allowdisplaybreaks  		
	\theoremstyle{plain}
\newtheorem{thm}{Theorem}[section]

\theoremstyle{plain}
\newtheorem{lem}[thm]{Lemma}

\theoremstyle{definition}
\newtheorem{defn}[thm]{Definition}

\theoremstyle{remark}
\newtheorem{rem}[thm]{Remark}

\theoremstyle{plain}

\theoremstyle{plain}
\newtheorem{cor}[thm]{Corollary}

  \theoremstyle{plain}
  \newtheorem*{cor*}{Corollary}		
\global\long\def\norm#1{\left\Vert #1\right\Vert }
\global\long\def\parenth#1{\left(\vphantom{#1}\right.\!\!#1\!\!\left.\vphantom{#1}\right)}
\global\long\def\curly#1{\left\{  #1\right\}  }
\global\long\def\set#1#2{\left\{  \vphantom{#1\ \vrule\ #2}\right.\!\!#1\ \vrule\ \linebreak[3]#2\!\!\left.\vphantom{#1\ \vrule\ #2}\right\}  }
\global\long\def\abs#1{\left\vert #1\right\vert }

\global\long\def\duality#1#2{\left\langle #1\,\vline width1pt\,#2\right\rangle }
\global\long\def\Duality#1#2{\left\llangle #1\,\vline width1pt\,#2\right\rrangle }

\global\long\def\R{\mathbb{R}}

\global\long\def\N{\mathbb{N}}

\global\long\def\closedball#1{\textup{\textbf{B}}_{#1}}

\newcommandx\closure[2][usedefault, addprefix=\global, 1=]{\overline{{#2}^{#1}}}

\global\long\def\convexhull#1{\textup{co}#1}
\global\long\def\closedconvexhull#1{\textup{\textbf{co}}#1}

\newcommand{\bottompolarsymbol}{
	\odot%
}

\global\long\def\bottompolar#1{#1^{\bottompolarsymbol}}

\global\long\def\bibottompolar#1{#1^{\bottompolarsymbol\bottompolarsymbol}}

\global\long\def\l#1{\ell^{#1}}
\global\long\def\conv{\textbf{c}}

\global\long\def\omegadomain#1{\parenth{\Omega,#1}}

\global\long\def\constants{\Xi}
\global\long\def\allconstants{\constants_{\infty}}
\global\long\def\oneconstants{\constants_{1}}
\global\long\def\summables{\Sigma}
\global\long\def\zerosummables{\summables_{0}}
\global\long\def\onesummables{\summables_{1}}

\global\long\def\directsum{\bigoplus}

\global\long\def\finite#1{\mathcal{F}\parenth{#1}}

\global\long\def\constsymb{\textbf{const}}
\global\long\def\constmap{\constsymb}

\global\long\def\Realpart{\textup{Re}\,}
\global\long\def\selfadjoint#1{#1_{\textup{sa}}}

\begin{document}
		\newcommand{\doctitle}{Geometric duality theory of cones\\ in dual pairs of vector spaces}
\newcommand{\shortdoctitle}{Geometric duality theory of cones}

\newcommand{\docabstract}{%
This paper will generalize what may be termed the ``geometric duality
theory'' of real pre-ordered Banach spaces which relates geometric
properties of a closed cone in a real Banach space, to geometric properties
of the dual cone in the dual Banach space. We show that geometric
duality theory is not restricted to real pre-ordered Banach spaces,
as is done classically, but can be extended to real Banach
spaces endowed with arbitrary collections of closed cones.

We define geometric notions of normality, conormality, additivity
and coadditivity for members of dual pairs of real vector spaces as
certain possible interactions between two cones and two convex convex
sets containing zero. We show that, thus defined, these notions are
dual to each other under certain conditions, i.e., for a dual pair
of real vector spaces $(Y,Z)$, the space $Y$ is normal (additive)
if and only if its dual $Z$ is conormal (coadditive) and vice versa.
These results are set up in a manner so as to provide a framework
to prove results in the geometric duality theory of cones in real
Banach spaces. As an example of using this framework, we generalize
classical duality results for real Banach spaces pre-ordered by a
single closed cone, to real Banach spaces endowed with an arbitrary
collections of closed cones.

As an application, we analyze some of the geometric properties of
naturally occurring cones in C{*}-algebras and their duals.
}

		\IfStrEqCase{\docclass}{%
			{amsart}{\title[Geometric duality theory of cones]{Geometric duality theory of cones\\ in dual pairs of vector spaces}}%
			{elsart}{\title{\doctitle}}%
		}
		\begin{abstract}%
			\docabstract%
		\end{abstract}
%
		\newcommand{\authornameMiek}{Miek Messerschmidt}
\newcommand{\emailMiek}{mmesserschmidt@gmail.com}

\newcommand{\addressLeiden}{%
	Mathematical Institute, 
	Leiden University, 
	P.O. Box 9512, 
	\mbox{2300 RA} Leiden, 
	The Netherlands
}

\newcommand{\addressNWU}{
	Unit for BMI,
	North-west University,
	Private Bag X6001,
	Potchefstroom,
	South Africa,
	2520
}

\newcommand{\authornameMarcel}{Marcel de Jeu}
\newcommand{\emailMarcel}{mdejeu@math.leidenuniv.nl}

\newcommand{\makeamsbio}{%
	%
	\author{\authornameMiek}%
	\address{\authornameMiek, \addressNWU}%
	\email{\emailMiek}
}

\newcommand{\makeelsbio}{%
	\address[nwu]{\addressNWU}

	
	\author[nwu]{\authornameMiek\corref{cor1}}
	\ead{\emailMiek}
	\cortext[cor1]{Corresponding author}

}

		\newcommand{\subjectClassesForThisPaper}{%
		Primary: 46A20; 
		Secondary: %
			46B10\subjclasssep
			46B20\subjclasssep
			46A40\subjclasssep
			46B40\subjclasssep
			46L05\ignorespaces   	
}

\newcommand{\keywordsForThisPaper}{
	cone\keywordssep 
	geometric duality theory\keywordssep 
	Banach space\keywordssep 
	normality\keywordssep 
	conormality\keywordssep 
	additivity\keywordssep 
	coadditivity\keywordssep 
	C*-algebras\keywordssep 
}

\newcommand{\makeamsothermeta}{%
	\newcommand{\keywordssep}{, }
	\newcommand{\subjclasssep}{, }
	
	\keywords{\keywordsForThisPaper}
	\subjclass[2010]{\subjectClassesForThisPaper}
}

\newcommand{\makeelsothermeta}{%
	\newcommand{\keywordssep}{\sep }
	\newcommand{\subjclasssep}{\sep }

	\begin{keyword}%
		\keywordsForThisPaper%
		\MSC[2010]{\subjectClassesForThisPaper}%
	\end{keyword}%
}
%
		\IfStrEqCase{\docclass}{%
			{amsart}{\makeamsbio}%
			{elsart}{\makeelsbio}%
		}
%
		\IfStrEqCase{\docclass}{%
			{amsart}{\makeamsothermeta}%
			{elsart}{\makeelsothermeta}%
		}	
		\maketitle 
\section{Introduction}

The goal of this paper is to provide a general framework for proving
results in what may be termed the ``geometric duality theory'' of
cones in real Banach spaces. We will prove The General Duality Theorems
(Theorems \ref{thm:Normality-Duality} and \ref{thm:Additivity-Duality})
for dual pairs of real vector spaces and, as an application, we generalize
classical results for real pre-ordered Banach spaces (cf. Theorems
\ref{thm:classical-normality-duality} and \ref{thm:classical-additivity-duality})
to the context of real Banach spaces endowed with arbitrary collections
of closed cones (cf. Corollaries \ref{cor:general-banach-space-normality-duality-interpreted}
and \ref{cor:general-banach-space-additivity-duality-interpreted}). 

We begin with some motivating historical remarks:

And\^o's Theorem \cite[Lemma 1]{Ando}, a fundamental result in the
geometric theory of real pre-ordered  Banach spaces, states: 
\begin{thm}
\label{thm:Andos-theorem}Let $X$ be a real Banach space, pre-ordered
by a closed cone $X_{+}$ (in the sense that $x\leq y$ means $y\in x+X_{+}$).
If the cone $X_{+}$ generates $X$, i.e., $X=X_{+}-X_{+}$, then
there exists a constant $\alpha\geq1$, such that every $x\in X$
can be written as $x=a-b$ with $a,b\in X_{+}$ and $\max\{\norm a,\norm b\}\leq\alpha\norm x$.
\end{thm}
The geometric property given by the conclusion of And\^o's Theorem
is what is termed a `conormality property%
\footnote{The term `conormality' is due to Walsh \cite{Walsh}. Historically
this notion appears under various names in the literature, of which
`$\alpha$-generating' or `boundedly generating' is quite common.%
}' (cf. property (2)(a) in Theorem \ref{thm:classical-normality-duality}
for another example). 

It has long been known that there exists a geometric duality theory
for such relationships between norms and closed cones in real pre-ordered
Banach spaces. Loosely speaking, there is a dual notion to conormality,
called `normality%
\footnote{The term `normality' is due to Krein \cite{KreinNormality}.%
}' (cf. property~(1)(a) in Theorem \ref{thm:classical-normality-duality}
for a specific example of a normality property). Normality is `dual'
to conormality in the sense that a real pre-ordered Banach space has
a conormality property if and only if its dual has a corresponding
normality property, and vice versa. Many such dual pairs of normality--
and conormality properties have been discovered for real pre-ordered
Banach spaces. A fairly complete inventory of normality and conormality
properties and their relationships as dual properties is given with
full references in \cite[Definition 3.1, Theorem 3.7]{MesserschmidtNormalityOfSpacesOfOperators}. 

The following is a representative sample of duality results in this
vein. To the author's knowledge, first proofs of this particular result
date back to Grosberg and Krein \cite{GrossbergKrein}, and Ellis
\cite{Ellis}: 
\begin{thm}
\label{thm:classical-normality-duality}Let $\alpha\geq1$. Let $X$
be real Banach space, pre-ordered by a closed cone $X_{+}$. Let the
dual space $X'$ be pre-ordered by the dual cone $X_{+}':=\set{\phi\in X'}{\phi(X_{+})\subseteq\R_{\geq0}}$, where $\R_{\geq0}:=\{\lambda\in\R\mid\lambda\geq0\}$. 
\begin{enumerate}
\item The following are equivalent:

\begin{enumerate}
\item If $x,a,b\in X$ satisfy $a\leq x\leq b$, then $\norm x\leq\alpha\max\{\norm a,\norm b\}$.
\item For every $\phi\in X'$, there exist $\rho,\psi\in X_{+}'$ such
that $\phi=\rho-\psi$ and $\norm{\rho}+\norm{\psi}\leq\alpha\norm{\phi}$.
\end{enumerate}
\IfStrEqCase{\docclass}{{amsart}{}{elsart}{ \pagebreak }}
\item The following are equivalent:

\begin{enumerate}
\item For every $x\in X$ and $\beta>\alpha$, there exist $a,b\in X_{+}$
such that $x=a-b$ and $\norm a+\norm b\leq\beta\norm x$.
\item If $\rho,\phi,\psi\in X'$ satisfy $\rho\leq\phi\leq\psi$,
then $\norm{\phi}\leq\alpha\max\{\norm{\rho},\norm{\psi}\}$.
\end{enumerate}
\end{enumerate}
\end{thm}
Interest in normality and conormality properties is usually traced
to their influence on the order structure of the bounded operators
between real pre-ordered Banach spaces (cf. \cite[Section 3.4]{MesserschmidtNormalityOfSpacesOfOperators}
and \cite[Section 1.7]{BattyRobinson}).

Further results in the geometric duality theory of real pre-ordered
Banach spaces are the duality results between so-called `additivity
properties' (e.g., property (1)(a) in Theorem \ref{thm:classical-additivity-duality}
below) and `coadditivity properties%
\footnote{Historically, this property is called `directedness'. We prefer the
term `coadditivity' as mnemonic for illustrating the property being
dual to notion of `additivity', which parallels that of the terms
`normality' and `conormality'.%
}' (e.g., property (2)(a) in Theorem \ref{thm:classical-additivity-duality}
below). To the author's knowledge, Theorem \ref{thm:classical-additivity-duality}
was first established by Asimow and Ellis \cite{Asimow}, \cite[Theorems 3.5 and 3.8]{AsimowEllis}.
Wong and Ng also established a related result in \cite[Lemmas 9.24 and 9.25]{WongNg}.
\begin{thm}
\label{thm:classical-additivity-duality}Let $\alpha\geq1$ and $n\in\N$.
Let $X$ be a real Banach space, pre-ordered by a closed cone $X_{+}$.
Let the dual space $X'$ be pre-ordered by the dual cone $X_{+}':=\set{\phi\in X'}{\phi(X_{+})\subseteq\R_{\geq0}}$, where $\R_{\geq0}:=\{\lambda\in\R\mid\lambda\geq0\}$. 
\IfStrEqCase{\docclass}{{amsart}{\pagebreak }{elsart}{ }}
\begin{enumerate}
\item The following are equivalent:

\begin{enumerate}
\item If $\{x_{j}\}_{j=1}^{n}\subseteq X_{+}$, then $\sum_{j=1}^{n}\norm{x_{j}}\leq\alpha\norm{\sum_{j=1}^{n}x_{j}}$.	
\item For any $\{\phi_{j}\}_{j=1}^{n}\subseteq X'$, there exists some $\varphi\in X'$
such that $\phi_{j}\leq\varphi$ for all $j\in\{1,\ldots,n\}$ and
$\norm{\varphi}\leq\alpha\max_{j}\{\norm{\phi_{j}}\}$.
\end{enumerate}
\item The following are equivalent:

\begin{enumerate}
\item For any $\{x_{j}\}_{j=1}^{n}\subseteq X$ and $\beta>\alpha$, there
exists some $y\in X$ such that $x_{j}\leq y$ for all $j\in\{1,\ldots,n\}$
and $\norm y\leq\beta\max_{j}\{\norm{x_{j}}\}$.
\item If $\{\phi_{j}\}_{j=1}^{n}\subseteq X_{+}'$, then $\sum_{j=1}^{n}\norm{\phi_{j}}\leq\alpha\norm{\sum_{j=1}^{n}\phi_{j}}$.
\end{enumerate}
\end{enumerate}
\end{thm}

In \cite{WicksteadCompact} Wickstead established a connection between
additivity and coadditivity properties and order-boundedness of norm-precompact
sets in real pre-ordered Banach spaces.

Recently, in \cite{deJeuMesserschmidtOpenMapping} it was shown that
And\^o's Theorem (Theorem \ref{thm:Andos-theorem}, above) is a specific
case of a more general result \cite[Theorem 4.1]{deJeuMesserschmidtOpenMapping}:
\begin{thm}
\label{thm:general-ando-theorem}Let $X$ be a real or complex Banach
space and $\curly{C_{\omega}}_{\omega\in\Omega}$ a collection of
closed cones in $X$. If the collection of cones $\curly{C_{\omega}}_{\omega\in\Omega}$
generate $X$, i.e., if every $x\in X$ can be written as $x=\sum_{\omega\in\Omega}c_{\omega}$,
where $c_{\omega}\in C_{\omega}$ for all $\omega\in\Omega$, and
$\sum_{\omega\in\Omega}\norm{c_{\omega}}<\infty$, then there exists
a constant $\alpha\geq1$ such that every $x\in X$ can be written
as $x=\sum_{\omega\in\Omega}c_{\omega}$, where $c_{\omega}\in C_{\omega}$
for all $\omega\in\Omega$, and $\sum_{\omega\in\Omega}\norm{c_{\omega}}\leq\alpha\norm x$.\end{thm}
\begin{rem}
As an aside, we note that \cite[Theorem 4.1]{deJeuMesserschmidtOpenMapping},
in fact, shows that this decomposition can not only be done in a bounded
manner, but also (through an application of Michael's Selection Theorem)
in a continuous manner: For all $\omega\in\Omega$, the maps $x\mapsto c_{\omega}$
may be chosen so as to be positively homogeneous, bounded and continuous.
\end{rem}
In comparing And\^o's Theorem and Theorem \ref{thm:general-ando-theorem},
that the two cones $X_{+}$ and $-X_{+}$ in a real pre-ordered Banach
space $X$ are related by a minus sign is pure coincidence. What is
at play here is how a collection of cones in a Banach space interact,
where, in a real pre-ordered Banach space, we merely happen to study
interactions of collections of cones comprised of copies of $X_{+}$
and $-X_{+}$. As Theorem \ref{thm:general-ando-theorem} shows, a
version of And\^o's Theorem still holds regardless of whether the
cones are related by minus signs, granted that the entire collection
of cones generates the space. 

Given the analogous form of the conclusion of Theorem \ref{thm:general-ando-theorem}
to the classical conormality property (2)(a) from Theorem \ref{thm:classical-normality-duality},
a natural question to ask is whether classical geometric duality theory
for real pre-ordered Banach spaces (e.g., Theorems \ref{thm:classical-normality-duality}
and \ref{thm:classical-additivity-duality}) carries over to the more
general situation of Theorem \ref{thm:general-ando-theorem}. Establishing
this is the goal of the rest of this paper.

We follow a purely geometric route to this end. We give four general
geometric conditions (listed below) on interactions between two cones
$C$ and $D$ and two convex sets $B_{1}$ and $B_{2}$ containing
zero in a real vector space $Y$. The space $Y$ will be said to be
\emph{normal}, \emph{additive}, \emph{conormal} or \emph{coadditive}
with respect to $(C,D,B_{1},B_{2})$, if respectively (1), (2), (3)
or (4) in the following list of inclusions hold:
\begin{enumerate}
\item $(B_{2}+C)\cap D\subseteq B_{1}$
\item $(B_{2}\cap C)+D\subseteq B_{1}$
\item $B_{1}\subseteq(B_{2}\cap C)+D$
\item $B_{1}\subseteq(B_{2}+C)\cap D$ 
\end{enumerate}
Depending on the choice of $Y$ and the sets $C,D,B_{1}$ and $B_{2}$,
the properties listed in Theorems \ref{thm:classical-normality-duality}
and \ref{thm:classical-additivity-duality} can be shown to be equivalent
to one of the above set inclusions holding true, e.g., a real pre-ordered
Banach space $X$ satisfies the normality property (1)(a) from Theorem
\ref{thm:classical-normality-duality} if and only if the $\l{\infty}$-direct
sum $Y:=X\oplus_{\infty}X$ is normal with respect to $(C,D,B_{1},B_{2})$,
with $C:=X_{+}\oplus_{\infty}(-X_{+})$, $D:=\set{(x,x)\in X\oplus_{\infty}X}{x\in X}$,
$B_{1}:=\set{(x,x)\in X\oplus_{\infty}X}{x\in X,\ \norm x\leq\alpha}$
and $B_{2}:=\set{(x,y)\in X\oplus_{\infty}X}{\norm{(x,y)}_{\infty}\leq1}$.
Similarly, the remaining properties from Theorems \ref{thm:classical-normality-duality}
and \ref{thm:classical-additivity-duality} can be shown to be equivalent
to one of the above set inclusions for specific choices of $Y$ and
the sets $C,D,B_{1}$ and $B_{2}$. Furthermore, this definition is
also general enough to encompass more general properties than those
that occur in real pre-ordered Banach spaces, e.g., the conclusion
of Theorem \ref{thm:general-ando-theorem} is equivalent to the $\l 1$-direct
sum of $\abs{\Omega}$ copies of $X$, $Y:=\l 1\omegadomain X$, being
conormal with respect to $(C,D,B_{1},B_{2})$, where $C:=\bigoplus_{\omega\in\Omega}C_{\omega}\subseteq\l 1\omegadomain X$,
$D:=\set{\xi\in\l 1\omegadomain X}{\sum_{\omega\in\Omega}\xi_{\omega}=0}$,
$B_{1}:=\set{\xi\in\l 1\omegadomain X}{\norm{\sum_{\omega\in\Omega}\xi_{\omega}}\leq1}$
and $B_{2}:=\set{\xi\in\l 1\omegadomain X}{\norm{\xi}_{1}\leq\alpha}$.

We show, using Lemma \ref{lem:elementary-normality-polars} (an elementary
application of the one-sided polar calculus) and the conditions outlined
in the General Duality Theorems (Theorems \ref{thm:Normality-Duality}
and \ref{thm:Additivity-Duality}), that normality (additivity) and
conormality (coadditivity), as defined above, are dual notions. I.e.,
for a dual pair of real vector spaces $(Y,Z)$ with $C,D,B_{1}$ and
$B_{2}$ subsets of $Y$ as above, we give conditions under which
$Y$ is normal (additive) with respect to $(C,D,B_{1},B_{2})$ if
and only if $Z$ is conormal (coadditive) with respect to the one-sided
polars $(\bottompolar C,\bottompolar D,\bottompolar{B_{1}},\bottompolar{B_{2}})$,
and vice versa.

The General Duality Theorems (Theorems \ref{thm:Normality-Duality}
and \ref{thm:Additivity-Duality}) hence provide a general framework
for proving results in the geometric duality theory for cones in real
Banach spaces (and real pre-ordered Banach spaces as a specific case).
To apply these results, one is only required to find a relevant dual
pair of real vector spaces $(Y,Z)$ and sets $C$, $D$, $B_{1}$
and $B_{2}$ in $Y$, calculate their one-sided polars in $Z$, and
verify the hypotheses of one of the General Duality Theorems (Theorems
\ref{thm:Normality-Duality} and \ref{thm:Additivity-Duality}) to
obtain a duality result in the vein of Theorems \ref{thm:classical-normality-duality}
and \ref{thm:classical-additivity-duality}. 

Following this program, generalizing Theorems \ref{thm:classical-normality-duality}
and \ref{thm:classical-additivity-duality} above to real Banach spaces
endowed with arbitrary collections of closed cones then becomes a
matter of routine (cf. Section \ref{sec:application-geometric-duality-in-Banach-spaces}).
We explicitly state our Corollaries \ref{cor:general-banach-space-normality-duality-interpreted}
and \ref{cor:general-banach-space-additivity-duality-interpreted}
of Theorems~\ref{thm:general-banach-space-normality-duality} and
\ref{thm:general-banach-space-additivity-duality}, which directly
generalize Theorems \ref{thm:classical-normality-duality} and \ref{thm:classical-additivity-duality}
above:

\setcounter{section}{4}
\setcounter{thm}{6}
\begin{cor}
Let $X$ be a real Banach space and $\{C_{\omega}\}_{\omega\in\Omega}$
a collection of closed cones in $X$ and $\alpha\geq1$.
\begin{enumerate}
\item The following are equivalent:

\begin{enumerate}
\item If $\xi\in\conv\omegadomain X$ and $x\in\bigcap_{\omega\in\Omega}\parenth{\xi_{\omega}+C_{\omega}}$,
then $\norm x\leq\alpha\norm{\xi}_{\infty}$.
\item For every $\phi\in X'$, there exists an element $\eta\in\directsum_{\omega\in\Omega}\bottompolar{C_{\omega}}\subseteq\l 1\omegadomain{X'}$
such that $\phi=\sum_{\omega\in\Omega}\eta_{\omega}$ and $\norm{\eta}_{1}\leq\alpha\norm{\phi}$.
\end{enumerate}
\item The following are equivalent:

\begin{enumerate}
\item For any $x\in X$ and $\beta>\alpha$, there exists an element $\xi\in\directsum_{\omega\in\Omega}C_{\omega}\subseteq\l 1\omegadomain X$
such that $x=\sum_{\omega\in\Omega}\xi_{\omega}$ and $\norm{\xi}_{1}\leq\beta\norm x$.
\item If $\eta\in\l{\infty}\omegadomain{X'}$ and $\phi\in\bigcap_{\omega\in\Omega}\parenth{\eta+\bottompolar{C_{\omega}}}$,
then $\norm{\phi}\leq\alpha\norm{\eta}_{\infty}$.
\end{enumerate}
\end{enumerate}

\noindent If, in addition, $\abs{\Omega}<\infty$ and $1\leq p,q\leq\infty$,
with $p^{-1}+q^{-1}=1$, then:
\begin{enumerate}
\item [\listitemhack{3}]The following are equivalent:

\begin{enumerate}
\item If $\xi\in\l p\omegadomain X$ and $x\in\bigcap_{\omega\in\Omega}\parenth{\xi_{\omega}+C_{\omega}}$,
then $\norm x\leq\alpha\norm{\xi}_{p}$.
\item For every $\phi\in X'$, there exists an element $\eta\in\directsum_{\omega\in\Omega}\bottompolar{C_{\omega}}\subseteq\l q\omegadomain{X'}$
such that $\phi=\sum_{\omega\in\Omega}\eta_{\omega}$ and $\norm{\eta}_{q}\leq\alpha\norm{\phi}$.
\end{enumerate}
\IfStrEqCase{\docclass}{{amsart}{\pagebreak }{elsart}{ }}
\item [\listitemhack{4}]The following are equivalent:

\begin{enumerate}
\item For every $x\in X$ and $\beta>\alpha$, there exists an element $\xi\in\directsum_{\omega\in\Omega}C_{\omega}\subseteq\l p\omegadomain X$
such that $x=\sum_{\omega\in\Omega}\xi_{\omega}$ and $\norm{\xi}_{p}\leq\beta\norm x$.
\item If $\eta\in\l q\omegadomain{X'}$ and $\phi\in\bigcap_{\omega\in\Omega}\parenth{\eta+\bottompolar{C_{\omega}}}$,
then $\norm{\phi}\leq\alpha\norm{\eta}_{q}$.
\end{enumerate}
\end{enumerate}
\end{cor}
{}\setcounter{thm}{8}
\begin{cor}
Let $X$ be a real Banach space and $\{C_{\omega}\}_{\omega\in\Omega}$
a collection of closed cones in $X$ and $\alpha\geq1$.
\begin{enumerate}
\item The following are equivalent:

\begin{enumerate}
\item If $\xi\in\directsum_{\omega\in\Omega}C_{\omega}\subseteq\l 1\omegadomain X$,
then $\norm{\xi}_{1}\leq\alpha\norm{\sum_{\omega\in\Omega}\xi_{\omega}}$.
\item For every $\eta\in\l{\infty}\omegadomain{X'}$, there exists some
$\phi\in\bigcap_{\omega\in\Omega}\parenth{\eta_{\omega}-\bottompolar{C_{\omega}}}$
with $\norm{\phi}\leq\alpha\norm{\eta}_{\infty}$.
\end{enumerate}
\IfStrEqCase{\docclass}{{amsart}{}{elsart}{ \pagebreak }}
\item The following are equivalent:

\begin{enumerate}
\item For every $\xi\in\conv\omegadomain X$ and $\beta>\alpha$, there
exists some $x\in\bigcap_{\omega\in\Omega}\parenth{\xi_{\omega}-C_{\omega}}$
with $\norm x\leq\beta\norm{\xi}_{\infty}$.
\item If $\eta\in\directsum_{\omega\in\Omega}\bottompolar{C_{\omega}}\subseteq\l 1\omegadomain{X'}$,
then $\norm{\eta}_{1}\leq\alpha\norm{\sum_{\omega\in\Omega}\eta_{\omega}}$.
\end{enumerate}
\end{enumerate}

\noindent If, in addition, $\abs{\Omega}<\infty$ and $1\leq p,q\leq\infty$,
with $p^{-1}+q^{-1}=1$, then:
\begin{enumerate}
\item [\listitemhack{3}]The following are equivalent:

\begin{enumerate}
\item If $\xi\in\directsum_{\omega\in\Omega}C_{\omega}\subseteq\l p\omegadomain X$,
then $\norm{\xi}_{p}\leq\alpha\norm{\sum_{\omega\in\Omega}\xi_{\omega}}$.
\item For every $\eta\in\l q\omegadomain{X'}$, there exists some $\phi\in\bigcap_{\omega\in\Omega}\parenth{\eta_{\omega}-\bottompolar{C_{\omega}}}$
with $\norm{\phi}\leq\alpha\norm{\eta}_{q}$.
\end{enumerate}
\item [\listitemhack{4}]The following are equivalent:

\begin{enumerate}
\item For every $\xi\in\l p\omegadomain X$ and $\beta>\alpha$, there exists
some $x\in\bigcap_{\omega\in\Omega}\parenth{\xi_{\omega}-C_{\omega}}$
with $\norm x\leq\beta\norm{\xi}_{p}$.
\item If $\eta\in\directsum_{\omega\in\Omega}\bottompolar{C_{\omega}}\subseteq\l q\omegadomain{X'}$,
then $\norm{\eta}_{q}\leq\alpha\norm{\sum_{\omega\in\Omega}\eta_{\omega}}$.
\end{enumerate}
\end{enumerate}
\end{cor}
\setcounter{section}{1}
\setcounter{thm}{4}

\bigskip

We note that there are situations where spaces endowed with multiple
cones sometimes arise quite naturally. A very simple example is that
of a C{*}-algebra $A$, when viewed as a real Banach space, which
is generated by the four cones $\{A_{+},-A_{+},iA_{+},-iA_{+}\}$,
where $A_{+}$ is the usual cone of positive elements of $A$. In
Section~\ref{sec:cstar} we include a brief application of our results
from the preceding sections to analyze the geometric properties of
naturally occurring cones in C{*}-algebras, culminating in Theorem
\ref{thm:cstar-order-results} below. 

\setcounter{section}{5}
\setcounter{thm}{2}
\begin{thm}
	Let $A$ be a C{*}-algebra and let $A'$ be its dual. Let $A_{+}$
    denote the cone of positive elements in $A$, and let $A_{+}'$ denote
    the cone of positive functionals on $A$, defined in the usual way
    for C{*}-algebras. Let $\selfadjoint A$ and $\selfadjoint A'$ respectively
    denote the real subspaces of self-adjoint elements in $A$ and $A'$.
    Then:
    \begin{enumerate}
    \item For $b_{1},b_{2},b_{3},b_{4}\in A$, if
    \begin{eqnarray*}
     &  & a\in\parenth{b_{1}+A_{+}+i\selfadjoint A}\cap\parenth{b_{2}-A_{+}+i\selfadjoint A}\\
     &  & \qquad\qquad\cap\parenth{b_{3}+iA_{+}+\selfadjoint A}\cap\parenth{b_{4}-iA_{+}+\selfadjoint A}
    \end{eqnarray*}
    then $\norm a\leq\max\curly{\norm{b_{1}},\norm{b_{2}}}+\max\curly{\norm{b_{3}},\norm{b_{4}}}$.
    \item For $b_{1},b_{2},b_{3},b_{4}\in A$, if
    \begin{eqnarray*}
     &  & a\in\parenth{b_{1}+A_{+}}\cap\parenth{b_{2}-A_{+}}\cap\parenth{b_{3}+iA_{+}}\cap\parenth{b_{4}-iA_{+}},
    \end{eqnarray*}
    then $\norm a\leq\max\curly{\norm{b_{1}},\norm{b_{2}}}+\max\curly{\norm{b_{3}},\norm{b_{4}}}$.
    \item For $a,b,c\in A$, if $a\leq b\leq c$, then $\norm b\leq2\max\curly{\norm a,\norm c}$.
    \item For $a,b,c\in\selfadjoint A$, if $a\leq b\leq c$, then $\norm b\leq\max\curly{\norm a,\norm c}$.
    \item For $\phi_{1},\phi_{2},\phi_{3},\phi_{4}\in A'$, if
    \begin{eqnarray*}
     &  & \varphi\in\parenth{\phi_{1}+A_{+}'+i\selfadjoint A'}\cap\parenth{\phi_{2}-A_{+}'+i\selfadjoint A'}\\
     &  & \qquad\qquad\cap\parenth{\phi_{3}+iA_{+}'+\selfadjoint A'}\cap\parenth{\phi_{4}-iA_{+}'+\selfadjoint A'}
    \end{eqnarray*}
    then $\norm{\varphi}\leq\sum_{j=1}^{4}\norm{\phi_{j}}$.
    \item For $\phi_{1},\phi_{2},\phi_{3},\phi_{4}\in A'$, if
    \begin{eqnarray*}
     &  & \varphi\in\parenth{\phi_{1}+A_{+}'}\cap\parenth{\phi_{2}-A_{+}'}\cap\parenth{\phi_{3}+iA_{+}'}\cap\parenth{\phi_{4}-iA_{+}'},
    \end{eqnarray*}
    then $\norm{\varphi}\leq\sum_{j=1}^{4}\norm{\phi_{j}}$.
    \item For $\rho,\phi,\psi\in A'$, if $\rho\leq\phi\leq\psi$, then
    $\norm{\phi}\leq2\parenth{\norm{\rho}+\norm{\psi}}$.
    \item For $\rho,\phi,\psi\in\selfadjoint A'$, if $\rho\leq\phi\leq\psi$,
    then $\norm{\phi}\leq\norm{\rho}+\norm{\psi}$.
    \end{enumerate}
\end{thm}
\setcounter{section}{1}
\setcounter{thm}{7}

\bigskip

We will now describe the structure of this paper:

In Section \ref{sec:Preliminaries} we give some preliminary notation
and results. We begin, in Section~\ref{sub:polar-calculus}, with
some elementary results on the polar calculus for dual pairs of real
vector spaces. Specifically, a result we will use numerous times is
Lemma \ref{lem:elementary-polar-results}(8) and (9), giving the exact
forms of the one-sided polar of the intersection (sum) of a (closed)
convex set containing zero and a cone. Further, in Section \ref{sub:convex-series},
we give basic definitions and a few basic results on convex analysis,
focusing on convex series. These results will be needed in the proofs
of the General Duality Theorems (Theorems \ref{thm:Normality-Duality}
and \ref{thm:Additivity-Duality}).

In Section \ref{sec:General-Geometric-duality-theory} we give general
definitions of normality, additivity, conormality and coadditivity
as the four set-inclusions above, and prove our two main results:
The General Duality Theorems (Theorems \ref{thm:Normality-Duality}
and \ref{thm:Additivity-Duality}).

Section \ref{sec:application-geometric-duality-in-Banach-spaces}
serves as an application of the section preceding it. By using the
General Duality Theorems (Theorems \ref{thm:Normality-Duality} and
\ref{thm:Additivity-Duality}), Corollaries \ref{cor:general-banach-space-normality-duality-interpreted}
and \ref{cor:general-banach-space-additivity-duality-interpreted}
are established. As already mentioned, results from this section follow
as a matter of routine verifications: For a real Banach space $X$
and arbitrary collection of closed cones $\curly{C_{\omega}}_{\omega\in\Omega}$,
we choose a related dual pair of real vector spaces $(Y,Z)$ and sets
$(C,D,B_{1},B_{2})$ in $Y$, compute their polars $(\bottompolar C,\bottompolar D,\bottompolar{B_{1}},\bottompolar{B_{2}})$
in $Z$, and verify the hypotheses of the General Duality Theorems
(Theorems \ref{thm:Normality-Duality} and \ref{thm:Additivity-Duality})
to obtain Theorems \ref{thm:general-banach-space-normality-duality}
and \ref{thm:general-banach-space-additivity-duality}. These theorems
are reformulated into Corollaries \ref{cor:general-banach-space-normality-duality-interpreted}
and \ref{cor:general-banach-space-additivity-duality-interpreted}
which generalize the classical results Theorems \ref{thm:classical-normality-duality}
and \ref{thm:classical-additivity-duality}.

Finally, in Section \ref{sec:cstar} we give a brief application of
our results from the preceding sections to C{*}-algebras. We prove
Theorem \ref{thm:cstar-order-results} on the geometric structure
of naturally occurring cones in C{*}-algebras and their duals.

\section{\label{sec:Preliminaries}Preliminary Definitions, Notation and Results}

All vector spaces in the rest of this paper are assumed to be over
$\R$. All topological vector spaces are assumed to be Hausdorff. 
Let $V$ be a topological vector space with topology $\tau$. Let
$A\subseteq V$. We denote the closure of $A$ by $\closure A$ (or
$\closure[\tau]A$ if confusion could arise). The (closed) convex
hull of $A$ will denoted by $\convexhull A$ ($\closedconvexhull A$).
The topological dual of $V$ will be denoted by $V'$, or by $(V,\tau)'$
if confusion could arise. A non-empty subset $C\subseteq V$ will
be called a \emph{cone} if $C+C\subseteq C$ and $\lambda C\subseteq C$
for all $\lambda\geq0$. If $C\subseteq V$ is a cone, we define its
\emph{dual cone by $C':=\set{\phi\in V'}{\phi(C)\subseteq\R_{\geq0}}$}, where we denote the non-negative real numbers by $\R_{\geq0}:=\{\lambda\in\R\mid\lambda\geq0\}$.

\subsection{\label{sub:polar-calculus}The polar calculus}

The current section will give some notation and basic results regarding
the one-sided polar calculus. Lemma \ref{lem:elementary-polar-results}(8)
and (9) on the one-sided polars of sums (intersections) of (closed)
convex sets and cones will be used numerous times in subsequent sections.

Let $Y$ and $Z$ be real vector spaces and $\duality{\cdot}{\cdot}:Y\times Z\to\R$
a bilinear map such that $\set{\duality{\cdot}z}{z\in Z}$ and $\set{\duality y{\cdot}}{y\in Y}$
separate the points of $Y$ and $Z$ respectively. We will then call
$(Y,Z)$ a \emph{dual pair}, and the map $\duality{\cdot}{\cdot}:Y\times Z\to\R$
a \emph{duality}. Unless otherwise mentioned, $Y$ and $Z$ will be
assumed to be respectively endowed with the weak topology (denoted
$\sigma(Y,Z)$), and the weak{*} topology (denoted $\sigma(Z,Y)$).
It is a well-known fact that $Y'=Z$ and $Z'=Y$ (cf. \cite[Theorem 5.93]{AliprantisBorder}). 

Let $(Y,Z)$ be a dual pair, $A\subseteq Y$ and $B\subseteq Z$.
We define the \emph{one-sided polar }of $A$ and $B$ respectively
by $\bottompolar A:=\set{z\in Z}{\duality az\leq1,\ \forall a\in A}$
and $\bottompolar B:=\set{y\in Y}{\duality yb\leq1,\ \forall b\in B}$. 

We state the following elementary properties of one-sided polars:
\begin{lem}
\label{lem:elementary-polar-results}Let $(Y,Z)$ be a dual pair,
$A$ a non-empty subset of $Y$, $C\subseteq Y$ a cone, and $\{A_{i}\}_{i\in I}$
a collection of non-empty subsets of $Y$. Then,
\begin{enumerate}
\item The set $\bottompolar A$ is closed, convex and contains zero.
\item $A\subseteq B$ implies $\bottompolar A\supseteq\bottompolar B$.
\item For every $\lambda>0$, $\bottompolar{(\lambda A)}=\lambda^{-1}(\bottompolar A)$.
\item $\bottompolar{\parenth{\bigcup_{i\in I}A_{i}}}=\bigcap_{i\in I}\bottompolar{A_{i}}$.
\item $\bibottompolar A=\closedconvexhull{(A\cup\{0\})}$.
\item If, for every $i\in I$, $A_{i}$ is closed convex and contains zero,
then 
\[
\bottompolar{\parenth{\bigcap_{i\in I}A_{i}}}=\closedconvexhull{\parenth{\bigcup_{i\in I}\bottompolar{A_{i}}}}.
\]

\item $\bottompolar C\subseteq Z$ is a closed cone and $\bottompolar C=-C'$.
\item If $A$ is closed convex and contains zero and $C$ is closed, then
$\bottompolar{(A\cap C)}=\closure{(\bottompolar A+\bottompolar C)}$.
\item If $A$ is convex and contains zero, then $\bottompolar{(A+C)}=\bottompolar A\cap\bottompolar C$.
\end{enumerate}
\end{lem}
\begin{proof}
Proofs of the statements (1)--(4) are elementary and left as an exercise
for the reader; (5) is The Bipolar Theorem \cite[Theorem~5.103]{AliprantisBorder}.

We prove (6). It is clear that $\closedconvexhull{\parenth{\bigcup_{i\in I}\bottompolar{A_{i}}}}\subseteq\bottompolar{\parenth{\bigcap_{i\in I}A_{i}}}$.
We prove the reverse inclusion. Let $z\in\bottompolar{\parenth{\bigcap_{i\in I}A_{i}}}$,
but suppose that $z\notin\closedconvexhull{\parenth{\bigcup_{i\in I}\bottompolar{A_{i}}}}$.
Then by The Separation Theorem \cite[Corollary~IV.3.10]{Conway},
there exists some $y\in Y$ and $\alpha\in\R$ with $\duality y{\closedconvexhull{\parenth{\bigcup_{i\in I}\bottompolar{A_{i}}}}}<\alpha<\duality yz$.
Since, by (1), $0\in\closedconvexhull{\parenth{\bigcup_{i\in I}\bottompolar{A_{i}}}}$,
we have $\alpha>0$, and hence, by dividing, we may assume $\alpha=1$,
so that $\duality y{\closedconvexhull{\parenth{\bigcup_{i\in I}\bottompolar{A_{i}}}}}<1<\duality yz$.
We conclude that $y\in\bibottompolar{A_{i}}$ for all $i\in I$. Moreover,
each $A_{i}$ is closed, convex and contains zero, so that, by (6),
$y\in\bibottompolar{A_{i}}=A_{i}$ for all $i\in I$, and hence $y\in\bigcap_{i\in I}A_{i}$.
But $\duality yz>1$ contradicts the assumption that $z\in\bottompolar{\parenth{\bigcap_{i\in I}A_{i}}}$.

We prove (7). That $\bottompolar C$ is closed follows from (1). Let
$z\in\bottompolar C$, i.e., $\duality yz\leq1$ for all $y\in C$.
Since $\lambda C\subseteq C$ for all $\lambda\geq0$, we conclude
that we must have $\duality yz\leq0$ for all $y\in C$. It is now
clear that $\bottompolar C+\bottompolar C\subseteq\bottompolar C$,
$\lambda\bottompolar C\subseteq\bottompolar C$ for all $\lambda\geq0$,
and that $\bottompolar C=-C'$.

We prove (8). By (6), $\bottompolar{(A\cap C)}=\closedconvexhull{(\bottompolar A\cup\bottompolar C)}$.
Since both $\bottompolar A$ and $\bottompolar C$ contain zero, we
have $\bottompolar A\cup\bottompolar C\subseteq\bottompolar A+\bottompolar C$.
Furthermore, since $\closure{(\bottompolar A+\bottompolar C)}$ is
closed and convex, it is clear that $\closedconvexhull{(\bottompolar A\cup\bottompolar C)}\subseteq\closure{(\bottompolar A+\bottompolar C)}$.
We prove the reverse inclusion by showing $(\bottompolar A+\bottompolar C)\subseteq\bottompolar{(A\cap C)}=\closedconvexhull{(\bottompolar A\cup\bottompolar C)}$.
Indeed, if $a\in\bottompolar A$ and $c\in\bottompolar C$, then for
every $y\in A\cap C$, by (7), we have $\duality y{a+c}=\duality ya+\duality yc\leq1+0=1$,
so that $a+c\in\bottompolar{(A\cap C)}=\closedconvexhull{(\bottompolar A\cup\bottompolar C)}$.
Therefore $(\bottompolar A+\bottompolar C)\subseteq\closedconvexhull{(\bottompolar A\cup\bottompolar C)}\subseteq\closure{(\bottompolar A+\bottompolar C)}$,
and hence $\closedconvexhull{(\bottompolar A\cup\bottompolar C)}=\closure{(\bottompolar A+\bottompolar C)}$.

We prove (9). Keeping (7) in mind, $\bottompolar{(A+C)}\supseteq\bottompolar A\cap\bottompolar C$
follows. Since both $A$ and $C$ contain zero, we obtain $A\cup C\subseteq A+C$.
Then, by (2) and (4), $\bottompolar{\parenth{A+C}}\subseteq\bottompolar{\parenth{A\cup C}}=\bottompolar A\cap\bottompolar C$.
\end{proof}

\subsection{\label{sub:convex-series}Convex series}

The current subsection gives basic definitions and results concerning
a more general notion of convexity in topological vector spaces; particularly
sets that are well behaved with respect to the taking of convex series
of their elements (in contrast to \emph{finite} convex combinations
where topology does not come into play). 

The somewhat technical result, Lemma \ref{lem:basic-cs-results}(6),
will be an essential ingredient in parts of our General Duality Theorems
(Theorems \ref{thm:Normality-Duality}(2)(b) and \ref{thm:Additivity-Duality}(2)(b)).
The proofs of many classical results for real pre-ordered Banach spaces,
like Theorems~\ref{thm:classical-normality-duality} and \ref{thm:classical-additivity-duality},
rely on similar results (cf. \cite[Corollary 1.3.3]{AsimowEllis}
and \cite[Lemma 1.1.3]{BattyRobinson}). 

Our terminology follows that of Jameson's from \cite{Jameson,JamesonConvexSeries}.
The terminology of using the prefix ``cs'' (for convex series) is
fairly standard (cf. \cite{JamesonConvexSeries,Zalinescu}), although
the term ``$\sigma$-convexity'' also does occur (cf. \cite{BattyRobinson}). 
\begin{defn}
\label{def:sigma-convex}Let $V$ be a topological vector space with
topology $\tau$ and $A\subseteq V$. 
\begin{enumerate}
\item The set $A\subseteq V$ will be called \emph{$\tau$-pre-cs-compact},
if, for all sequences $\{a_{n}\}\subseteq A$ and $\{\lambda_{n}\}\subseteq\R_{\geq0}$
with $\sum_{n=1}^{\infty}\lambda_{n}=1$, the series $\sum_{n=1}^{\infty}\lambda_{n}a_{n}$
converges in the $\tau$-topology.
\item The set $A\subseteq V$ will be called \emph{$\tau$-cs-compact},
if, for all sequences $\{a_{n}\}\subseteq A$ and $\{\lambda_{n}\}\subseteq\R_{\geq0}$
with $\sum_{n=1}^{\infty}\lambda_{n}=1$, the series $\sum_{n=1}^{\infty}\lambda_{n}a_{n}$
converges to a point in $A$ in the $\tau$-topology. 
\item The set $A\subseteq V$ will be called \emph{$\tau$-cs-closed} if,
for sequences $\{a_{n}\}\subseteq A$ and $\{\lambda_{n}\}\subseteq\R_{\geq0}$
with $\sum_{n=1}^{\infty}\lambda_{n}=1$, convergence of the series
$\sum_{n=1}^{\infty}\lambda_{n}a_{n}$ in the $\tau$-topology implies
$\sum_{n=1}^{\infty}\lambda_{n}a_{n}\in A$. 
\end{enumerate}
If no confusion arises as to which topology on $V$ is meant, we will
merely say $A$ is pre-cs-compact, cs-compact or cs-closed\emph{.}
\end{defn}
The following results give basic properties of pre-cs-compact, cs-compact
and cs-closed sets. Most results are elementary and will be left as
exercises (references are however given). The result (5) below is
a slight generalization of \cite[Lemma 1.1.3]{BattyRobinson}.
\begin{lem}
\label{lem:basic-cs-results}Let $V$ be a topological vector space.
\begin{enumerate}
\item In $V$, every cs-compact set is both cs-closed and pre-cs-compact,
and every subset of a pre-cs-compact set is itself pre-cs-compact.
\item In $V$, the intersection of a cs-compact set with a cs-closed set
is again cs-compact.
\item In $V$, every closed convex set is cs-closed and every open convex
set is cs-closed.
\item If $A\subseteq V$ is cs-compact and $B\subseteq V$ is cs-closed,
then $\convexhull{(A\cup B)}$ and $A+B$ are cs-closed.
\item If the topology on $V$ is normable, $V$ is a Banach space if and
only if its closed unit ball is cs-compact.
\item Let $A\subseteq V$ be cs-closed and let $G\subseteq D\subseteq\overline{A}$.
If $G$ is pre-cs-compact such that, for every $r>0$ and $d\in D$,
the set $(d-rG)\cap A$ is non-empty, then $D\subseteq\alpha A$ for
all $\alpha>1$.
\end{enumerate}
\end{lem}
\begin{proof}
The assertions (1) and (2) follow immediately from the definitions.

Proof of the assertion (3) can be found in \cite[Proposition~1.2.1.(i)]{Zalinescu}.
The argument for closed convex sets is elementary. The argument for
open convex is slightly more involved and relies on The Separation
Theorem \cite[Theorem~IV.3.7]{Conway}.

An elementary argument will prove (4). Proof can be found in \cite[Theorem~A.2]{Jameson}.

To establish (5), it can be seen that absolutely convergent series
converge if and only if the closed unit ball is cs-compact. 

We prove (6). Let $y\in D$ and $r\in(0,1)$ be arbitrary. We inductively
define sequences $\curly{b_{n}}\subseteq D$ and $\curly{a_{n}}\subseteq A$
as follows: For any $n\in\mathbb{N}$, if $\set{a_{j}}{j=1,\ldots,n-1}\subseteq A$,
we define $b_{n}:=r^{-(n-1)}y-\sum_{j=1}^{n-1}r^{j-n}a_{j}$. If $b_{n}\in D$,
choosing $a_{n}\in\left(b_{n}-rG\right)\cap A\neq\emptyset$ then
yields 
\begin{eqnarray*}
D\supseteq G & \ni & r^{-1}b_{n}-r^{-1}a_{n}\\
 & = & r^{-1}\parenth{r^{-(n-1)}y-\sum_{j=1}^{n-1}r^{j-n}a_{j}}-r^{n-(n+1)}a_{n}\\
 & = & r^{-n}y-\sum_{j=1}^{n-1}r^{j-(n+1)}a_{j}-r^{n-(n+1)}a_{n}\\
 & = & r^{-n}y-\sum_{j=1}^{n}r^{j-(n+1)}a_{j}\\
 & = & b_{n+1}
\end{eqnarray*}
Since $b_{1}=y\in D$, we therefore obtain the sequences $\{b_{n}\}\subseteq D$
and $\{a_{n}\}\subseteq A$. 

Since $b_{n+1}\in G$ for all $n\in\N$, and $G$ is pre-cs-compact,
the series $r^{-1}(r-1)\sum_{n=1}^{\infty}r^{n}b_{n+1}$ converges,
and hence $r^{n}b_{n+1}\to0$. Because $r^{n}b_{n+1}=y-\sum_{j=1}^{n}r^{j-1}a_{j},$
we conclude that the series $\sum_{j=0}^{\infty}r^{j}a_{j+1}$ converges
to $y$. Since $\{a_{n}\}\subseteq A$ and $A$ is cs-closed, $(1-r)\sum_{j=0}^{\infty}r^{j}a_{j+1}$
converges a point in $A$ which must equal $(1-r)y$. We obtain $y\in(1-r)^{-1}A$,
and since $r\in(0,1)$ was chosen arbitrarily, the result follows.\end{proof}
\begin{lem}
\label{lem:intersection-of-dilations}Let $Y$ be a locally convex
space. If $A\subseteq Y$ is a closed convex set containing zero,
then 
\[
\bigcap_{\lambda>1}\lambda A=A.
\]
\end{lem}
\begin{proof}
Let $a\in A$. If $a=0$, then clearly $a\in\bigcap_{\lambda>1}\lambda A$.
Suppose that $a\neq0$. Let $\lambda>1$ be arbitrary, so that, by
convexity of $\lambda A$, $a=\lambda^{-1}(\lambda a)+(1-\lambda^{-1})0\in\lambda A$.
We conclude that $A\subseteq\bigcap_{\lambda>1}\lambda A$.

We prove the reverse inequality. Suppose $y\in\bigcap_{\lambda>1}\lambda A$
is not an element of $A$. By The Separation Theorem \cite[Theorem IV.3.9]{Conway},
there exists a functional $\phi\in Y'$ and $\alpha\in\R$ such that
$\phi\parenth a<\alpha<\phi(y)$ for all $a\in A$. Since $0\in A$,
we have that $\alpha>0$. Let $\lambda_{0}:=(2\alpha)^{-1}(\alpha+\phi(y))>1$.
For every $a\in A$, 
\[
\phi(\lambda_{0}a)=\frac{(\alpha+\phi(y))\phi(a)}{2\alpha}<\frac{(\alpha+\phi(y))\alpha}{2\alpha}=\frac{\alpha+\phi(y)}{2}<\phi(y).
\]
Therefore $\phi(y-\lambda_{0}a)>0$ for every $a\in A$. Hence $y\notin\lambda_{0}A$,
contradicting the assumption $y\in\bigcap_{\lambda>1}\lambda A$.
Therefore $\bigcap_{\lambda>1}\lambda A\subseteq A$. We conclude
that $\bigcap_{\lambda>1}\lambda A=A$.
\end{proof}

\section{\label{sec:General-Geometric-duality-theory}Geometric duality theory
for cones in dual pairs of vector spaces:\\A general framework}

In this section we define general notions of normality, additivity,
conormality and coadditivity as interactions of two cones with two
convex sets containing zero. Using our preliminary results from the
previous section we prove our main results: The General Duality Theorems
(Theorems \ref{thm:Normality-Duality} and \ref{thm:Additivity-Duality}). 
\begin{defn}
\label{def:Normality-Conormality-Additivity-Coadditivity}Let $Y$
be a vector space. Let $C,D\subseteq Y$ be cones and $B_{1},B_{2}\subseteq Y$
convex sets containing zero. 
\begin{enumerate}
\item We will say that $Y$ is \emph{normal} with respect to $(C,D,B_{1},B_{2})$
if 
\[
(B_{2}+C)\cap D\subseteq B_{1}.
\]

\item We will say that $Y$ is \emph{additive} with respect to $(C,D,B_{1},B_{2})$
if 
\[
(B_{2}\cap C)+D\subseteq B_{1}.
\]

\item We will say that $Y$ is \emph{conormal} with respect to $(C,D,B_{1},B_{2})$
if 
\[
B_{1}\subseteq(B_{2}\cap C)+D.
\]

\item We will say that $Y$ is \emph{coadditive} with respect to $(C,D,B_{1},B_{2})$
if 
\[
B_{1}\subseteq(B_{2}+C)\cap D.
\]

\end{enumerate}
\end{defn}
\begin{rem}
\label{rem:constant-passing-remark-in-normality-additivity}We note
that, if, for some $\alpha>0$, $Y$ has one of the above properties
with respect to $(C,D,B_{1},\alpha B_{2})$, then it has the same
property with respect to $(C,D,\alpha^{-1}B_{1},B_{2})$.
\end{rem}
Elementary applications of Lemma \ref{lem:elementary-polar-results}
yield the following result. 
\begin{lem}
\label{lem:elementary-normality-polars}Let $(Y,Z)$ be a dual pair
with $\sigma(Y,Z)$-closed cones $C,D\subseteq Y$ and $B_{1},B_{2}\subseteq Y$
$\sigma(Y,Z)$-closed convex sets containing zero.
\begin{enumerate}
\item If $Y$ is normal with respect to $(C,D,B_{1},B_{2})$, then $\bottompolar{B_{1}}\subseteq\closure{(\bottompolar{B_{2}}\cap\bottompolar C)+\bottompolar D}$.
\item If $Y$ is additive with respect to $(C,D,B_{1},B_{2})$, then $\bottompolar{B_{1}}\subseteq\closure{\parenth{\bottompolar{B_{2}}+\bottompolar C}}\cap\bottompolar D$.
\item If $Y$ is conormal with respect to $(C,D,B_{1},B_{2})$, then $\parenth{\bottompolar{B_{2}}+\bottompolar C}\cap\bottompolar D\subseteq\bottompolar{B_{1}}$.
\item If $Y$ is coadditive with respect to $(C,D,B_{1},B_{2})$, then $(\bottompolar{B_{2}}\cap\bottompolar C)+\bottompolar D\subseteq\bottompolar{B_{1}}$
\item If $Z$ is normal with respect to $(\bottompolar C,\bottompolar D,\bottompolar{B_{1}},\bottompolar{B_{2}})$,
then $B_{1}\subseteq\closure{B_{2}\cap C+D}.$
\item If $Z$ is additive with respect to $(\bottompolar C,\bottompolar D,\bottompolar{B_{1}},\bottompolar{B_{2}})$,
then $B_{1}\subseteq\closure{\parenth{B_{2}+C}}\cap D.$
\item If $Z$ is conormal with respect to $(\bottompolar C,\bottompolar D,\bottompolar{B_{1}},\bottompolar{B_{2}})$,
then $\parenth{B_{2}+C}\cap D\subseteq B_{1}.$
\item If $Z$ is coadditive with respect to $(\bottompolar C,\bottompolar D,\bottompolar{B_{1}},\bottompolar{B_{2}})$,
then $(B_{2}\cap C)+D\subseteq B_{1}.$
\end{enumerate}
\end{lem}
\begin{proof}
The results follow from elementary applications of Lemma \ref{lem:elementary-polar-results}
(and noting that for all subsets $K$ and $L$ of a topological space,
both $K\subseteq\closure K$ and $K\cap L\subseteq\closure K\cap L$
hold).
\end{proof}
With the above observation and the results established in the previous
section, the General Duality Theorems below now become fairly simple
verifications. 
																\IfStrEqCase{\docclass}{{amsart}{\vfill\pagebreak}}
\begin{thm}
\label{thm:Normality-Duality}(General duality between normality and
conormality) Let $(Y,Z)$ be a dual pair with $\sigma(Y,Z)$-closed
cones $C,D\subseteq Y$ and $B_{1},B_{2}\subseteq Y$ $\sigma(Y,Z)$-closed
convex sets containing zero. 
\begin{enumerate}
\item Of the statements (i) and (ii) below:\vspace*{1mm}

\begin{enumerate}
\item (ii) implies (i). 
\item If $\bottompolar{B_{2}}\cap\bottompolar C+\bottompolar D$ is $\sigma(Z,Y)$-closed,
then (i) and (ii) are equivalent.\vspace*{3mm}

\begin{enumerate}
\item $Y$ is normal with respect to $(C,D,B_{1},B_{2})$.
\item $Z$ is conormal with respect to $(\bottompolar C,\bottompolar D,\bottompolar{B_{1}},\bottompolar{B_{2}})$.\vspace*{3mm}
\end{enumerate}
\end{enumerate}
\item Of the statements (i)--(iii) below: \vspace*{1mm}

\begin{enumerate}
\item (ii) implies (iii). 
\item If $B_{2}\cap C+D$ is $\sigma(Y,Z)$-cs-closed and $B_{1}$ contains
a $\sigma(Y,Z)$-pre-cs-compact set $G$, such that, for every $r>0$
and $b\in B_{1}$, $(b-rG)\cap(B_{2}\cap C+D)\neq\emptyset$, then
(ii) and (iii) are equivalent . 
\item If $B_{2}\cap C+D$ is $\sigma(Y,Z)$-closed, then (i), (ii) and (iii)
are equivalent.\vspace*{3mm}

\begin{enumerate}
\item $Y$ is conormal with respect to $(C,D,B_{1},B_{2})$.
\item $Y$ is conormal with respect to $(C,D,B_{1},\lambda B_{2})$ for
all $\lambda>1$.
\item $Z$ is normal with respect to $(\bottompolar C,\bottompolar D,\bottompolar{B_{1}},\bottompolar{B_{2}})$.
\end{enumerate}
\end{enumerate}
\end{enumerate}
\end{thm}
\begin{proof}
We prove (1)(a). This is immediate from Lemma \ref{lem:elementary-normality-polars}.

We prove (1)(b). By (1)(a), it suffices to prove that (i) implies
(ii). By Lemma \ref{lem:elementary-normality-polars} above and the
assumption that $\bottompolar{B_{2}}\cap\bottompolar C+\bottompolar D$
is closed, $\bottompolar{B_{1}}\subseteq\closure{\bottompolar{B_{2}}\cap\bottompolar C+\bottompolar D}=\bottompolar{B_{2}}\cap\bottompolar C+\bottompolar D$.
We conclude that $Z$ is conormal with respect to $(\bottompolar C,\bottompolar D,\bottompolar{B_{1}},\bottompolar{B_{2}})$.

We prove (2)(a). By Lemma \ref{lem:elementary-normality-polars} and
Remark \ref{rem:constant-passing-remark-in-normality-additivity},
$\parenth{\bottompolar{B_{2}}+\bottompolar C}\cap\bottompolar D\subseteq\lambda(\bottompolar{B_{1}})$
holds for every $\lambda>1$. By Lemma \ref{lem:intersection-of-dilations},
$\bigcap_{\lambda>1}\lambda\bottompolar{B_{1}}=\bottompolar{B_{1}}$.
Therefore $\parenth{\bottompolar{B_{2}}+\bottompolar C}\cap\bottompolar D\subseteq\bigcap_{\lambda>1}\lambda\bottompolar{B_{1}}=\bottompolar{B_{1}}$,
and we conclude that $Z$ is normal with respect to $(\bottompolar C,\bottompolar D,\bottompolar{B_{1}},\bottompolar{B_{2}})$.

We prove (2)(b). By (2)(a) it suffices to prove that (iii) implies
(ii). By Lemma~\ref{lem:elementary-normality-polars} above $B_{1}\subseteq\closure{B_{2}\cap C+D}$.
Since $B_{2}\cap C+D$ is cs-closed and $B_{1}$ contains is a pre-cs-compact
set with the stated property, by Lemma \ref{lem:basic-cs-results},
for every $\lambda>1$, $B_{1}\subseteq\lambda(B_{2}\cap C+D)=(\lambda B_{2})\cap C+D$.
We conclude that $Y$ is conormal with respect to $(C,D,B_{1},\lambda B_{2})$
for all $\lambda>1$.

We prove (2)(c). By (2)(a) we have that (ii) implies (iii). Since
$B_{2}\subseteq\lambda B_{2}$ for all $\lambda>1$, we have $B_{1}\subseteq C\cap B_{2}+D\subseteq C\cap(\lambda B_{2})+D$,
and hence (i) implies (ii). If (iii) holds, then by Lemma \ref{lem:elementary-normality-polars}
above, and the assumption that $B_{2}\cap C+D$ is closed, $B_{1}\subseteq\closure{B_{2}\cap C+D}=B_{2}\cap C+D$,
so that (iii) implies (i). \end{proof}
\begin{thm}
\label{thm:Additivity-Duality}(General duality between additivity
and coadditivity) Let $(Y,Z)$ be a dual pair with $\sigma(Y,Z)$-closed
cones $C,D\subseteq Y$ and $B_{1},B_{2}\subseteq Y$ $\sigma(Y,Z)$-closed
convex sets containing zero. 
\begin{enumerate}
\item Of the statements (i) and (ii) below:\vspace*{1mm}

\begin{enumerate}
\item (ii) implies (i). 
\item If $\bottompolar{B_{2}}+\bottompolar C$ is $\sigma(Z,Y)$-closed,
then (i) and (ii) are equivalent.\vspace*{3mm}

\begin{enumerate}
\item $Y$ is additive with respect to $(C,D,B_{1},B_{2})$.
\item $Z$ is coadditive with respect to $(\bottompolar C,\bottompolar D,\bottompolar{B_{1}},\bottompolar{B_{2}})$.\vspace*{3mm}
\end{enumerate}
\end{enumerate}
\IfStrEqCase{\docclass}{{amsart}{\pagebreak }{elsart}{ }}
\item Of the statements (i)--(iii) below:\vspace*{1mm}

\begin{enumerate}
\item (ii) implies (iii). 
\item If $B_{2}+C$ is $\sigma(Y,Z)$-cs-closed and $B_{1}$ contains a
$\sigma(Y,Z)$-pre-cs-compact set $G$ such that, for every $b\in B_{1}$
and $r>0$, $\parenth{b-rG}\cap\parenth{B_{2}+C}\neq\emptyset$, then
(ii) and (iii) are equivalent .
\item If $B_{2}+C$ is $\sigma(Y,Z)$-closed, then (i), (ii) and (iii) are
equivalent.\vspace*{3mm}

\begin{enumerate}
\item $Y$ is coadditive with respect to $(C,D,B_{1},B_{2})$.
\item $Y$ is coadditive with respect to $(C,D,B_{1},\lambda B_{2})$ for
all $\lambda>1$.
\item $Z$ is additive with respect to $(\bottompolar C,\bottompolar D,\bottompolar{B_{1}},\bottompolar{B_{2}})$.
\end{enumerate}
\end{enumerate}
\end{enumerate}
\end{thm}
\begin{proof}
We prove (1)(a). This is immediate from Lemma \ref{lem:elementary-normality-polars}.

We prove (1)(b). By (1)(a) it suffices to show that (i) implies (ii).
By Lemma~\ref{lem:elementary-normality-polars}, $\bottompolar{B_{1}}\subseteq\closure{\parenth{\bottompolar{B_{2}}+\bottompolar C}}\cap\bottompolar D$.
Since $\bottompolar{B_{2}}+\bottompolar C$ is closed, $\bottompolar{B_{1}}\subseteq\parenth{\bottompolar{B_{2}}+\bottompolar C}\cap\bottompolar D$,
and we conclude that $Z$ is coadditive with respect to $(\bottompolar C,\bottompolar D,\bottompolar{B_{1}},\bottompolar{B_{2}})$.

We prove (2)(a). By Lemma \ref{lem:elementary-normality-polars} and
Remark \ref{rem:constant-passing-remark-in-normality-additivity},
$\parenth{\bottompolar{B_{2}}+\bottompolar C}\cap\bottompolar D\subseteq\lambda\bottompolar{B_{1}}$
for every $\lambda>1$. By Lemma \ref{lem:intersection-of-dilations},
$\parenth{\bottompolar{B_{2}}\cap\bottompolar C+\bottompolar D}\subseteq\bigcap_{\lambda>1}\lambda\parenth{\bottompolar{B_{1}}}=\bottompolar{B_{1}}$.
We conclude that $Z$ is additive with respect to $(\bottompolar C,\bottompolar D,\bottompolar{B_{1}},\bottompolar{B_{2}})$.

We prove (2)(b). By (2)(a) it suffices to prove that (iii) implies
(ii). By Lemma~\ref{lem:elementary-normality-polars}, $B_{1}\subseteq\closure{\parenth{B_{2}+C}}\cap D\subseteq\closure{\parenth{B_{2}+C}}$.
Since $B_{1}$ is assumed to contain a pre-cs-compact set with the
stated property, by Lemma \ref{lem:basic-cs-results}, $B_{1}\subseteq\lambda\parenth{B_{2}+C}=\lambda B_{2}+C$
holds for every $\lambda>1$. Now, because $B_{1}\subseteq\closure{\parenth{B_{2}+C}}\cap D\subseteq D$,
we see that $B_{1}\subseteq\parenth{\lambda B_{2}+C}\cap D$ holds
for every $\lambda>1$. We conclude that $Y$ is coadditive with respect
to $(C,D,B_{1},\lambda B_{2})$ for all $\lambda>1$. 

We prove (2)(c). By 2(a), (ii) implies (iii). Since $B_{2}\subseteq\lambda B_{2}$
for all $\lambda>1$, we have $B_{1}\subseteq(B_{2}+C)\cap D\subseteq(\lambda B_{2}+C)\cap D$
for all $\lambda>1$, so that (i) implies (ii). If (iii) holds, by
Lemma \ref{lem:elementary-normality-polars} and the hypothesis that
$B_{2}+C$ is closed, $B_{1}\subseteq\closure{\parenth{B_{2}+C}}\cap D=\parenth{B_{2}+C}\cap D$,
so that (iii) implies (i).\end{proof}
\begin{rem}
The most technical parts of the two theorems above is the reliance
on Lemma \ref{lem:basic-cs-results}(6) in the proofs of Theorems
\ref{thm:Normality-Duality}(2)(b) and \ref{thm:Additivity-Duality}(2)(b)
below, in which a duality result follows by ``paying an arbitrarily
small price'' in scaling up the set $B_{2}$. The necessity of Lemma
\ref{lem:basic-cs-results}(6) can be explained by the set $B_{2}$
in the previous two theorems not being $\sigma(Y,Z)$-compact in general.
If $B_{2}$ is $\sigma(Y,Z)$-compact, then Theorems \ref{thm:Normality-Duality}(2)(c)
and \ref{thm:Additivity-Duality}(2)(c) apply and Lemma \ref{lem:basic-cs-results}(6)
is not needed. The situation is likely best seen in the setting of Banach
spaces, e.g., Theorem \ref{thm:general-banach-space-normality-duality},
where $B_{2}$ is chosen to be a closed ball of a not-necessarily
reflexive Banach space and is hence is not necessarily weakly-compact.
In \cite[Examples 2.1.6 and 2.3.10]{AsimowEllis} non-reflexive examples
are given, showing that the conclusions of Theorems \ref{thm:Normality-Duality}(2)(b)
and \ref{thm:Additivity-Duality}(2)(b) are indeed the best possible.
\end{rem}

\section{\label{sec:application-geometric-duality-in-Banach-spaces}Application:
Geometric duality theory of cones in Banach spaces}

The current section is a routine application of the General Duality
Theorems (Theorems \ref{thm:Normality-Duality} and \ref{thm:Additivity-Duality})
from the previous section. Our main results in this section are Theorems~\ref{thm:general-banach-space-normality-duality}
and \ref{thm:general-banach-space-additivity-duality} and their reformulations
Corollaries~\ref{cor:general-banach-space-normality-duality-interpreted}
and \ref{cor:general-banach-space-additivity-duality-interpreted}.
These results generalize Theorems~\ref{thm:classical-normality-duality}
and \ref{thm:classical-additivity-duality} from the introduction.

We begin with some preliminary notation and definitions used in this
section:

For sets $A$ and $B$, by $B^{A}$ we will denote the set of all
functions from $A$ to $B$. Throughout this section $X$ will denote
an arbitrary real Banach space and $X'$ its topological dual. The
map $\duality{\cdot}{\cdot}:X\times X'\to\R$ will denote the usual
(evaluation) duality for the dual pair $(X,X')$. We will denote the
closed unit ball of a Banach space $X$ by $\closedball X$. 

For an index set $\Omega$ we will denote the directed set (ordered
by inclusion) of finite subsets of $\Omega$ by $\finite{\Omega}$.
For the sake of readability, the value of a function $\xi\in X^{\Omega}$
at some $\omega\in\Omega$ will be usually be denoted by $\xi_{\omega}$
instead of $\xi(\omega)$. For any function $f\in\R^{\Omega}$ and
$x\in X$, we define $f\otimes x\in X^{\Omega}$ by $\Omega\ni\omega\mapsto f(\omega)x$.
For $A\subseteq\Omega$ we will denote the characteristic function
of $A$ by $\chi_{A}$, and define $\delta_{\omega}:=\chi_{\{\omega\}}$
for all $\omega\in\Omega$. 

For $\xi\in X^{\Omega}$, by $\sum_{\omega\in\Omega}\xi_{\omega}$
we mean the norm-limit of the net $\curly{\sum_{\omega\in F}\xi_{\omega}}_{F\in\finite{\Omega}}$
(if it exists). For subspaces $Y\subseteq X^{\Omega}$ and $Z\subseteq X'^{\Omega}$,
if $Y\times Z\ni(\xi,\eta)\mapsto\sum_{\omega\in\Omega}\duality{\xi_{\omega}}{\eta_{\omega}}$
defines a duality on $(Y,Z)$, we will denote it by $\Duality{\cdot}{\cdot}:Y\times Z\to\R$.

We define the following classical direct sums of a Banach space:
\begin{defn}
Let $\Omega$ be an index set and $X$ a Banach space.
\begin{enumerate}
\item For $1\leq p<\infty$, by $\l p\omegadomain X$ we will denote the
subspace of $X^{\Omega}$ of all elements $\xi\in X^{\Omega}$ satisfying
$\sum_{\omega\in\Omega}\norm{\xi_{\omega}}^{p}<\infty$, with norm
$\norm{\xi}_{p}:=\parenth{\sum_{\omega\in\Omega}\norm{\xi_{\omega}}^{p}}^{1/p}$.
\item By $\l{\infty}\omegadomain X$ we will denote the subspace of $X^{\Omega}$
of all elements $\xi\in X^{\Omega}$ satisfying $\sup_{\omega\in\Omega}\norm{\xi_{\omega}}<\infty$,
with norm $\norm{\xi}_{\infty}:=\sup_{\omega\in\Omega}\norm{\xi_{\omega}}$.
\item By $\conv\omegadomain X$ we will denote the closed subspace of $\l{\infty}\omegadomain X$
of all elements $\xi\in\l{\infty}\omegadomain X$ for which there
exists some $x\in X$ such that, for every $\varepsilon>0$, there
exists some $F\in\finite{\Omega}$, with $\sup_{\omega\in\Omega\backslash F}\norm{\xi_{\omega}-x}<\varepsilon$.
\end{enumerate}
\end{defn}
We will use the folklore-result (cf. \cite{DiestelAbsolutelySummable})
that the duals of $\conv\omegadomain X$, $\l 1\omegadomain X$ and
$\l p\omegadomain X$ (for $1<p,q<\infty$ with $p^{-1}+q^{-1}=1$)
may be isometrically isomorphically identified with $\l 1\omegadomain{X'}$,
$\l{\infty}\omegadomain{X'}$ and $\l q\omegadomain{X'}$ respectively,
where evaluation is given by $\Duality{\cdot}{\cdot}$.
\begin{defn}
Let $X$ be a Banach space and $\Omega$ an index set. 
\begin{enumerate}
\item We define the \emph{canonical summation operator }$\Sigma:X^{\Omega}\to X\cup\{\infty\}$
as follows:
\[
\Sigma\xi:=\begin{cases}
\sum_{\omega\in\Omega}\xi_{\omega} & \mbox{If \ensuremath{\sum_{\omega\in\Omega}\xi_{\omega}\ }converges in norm in \ensuremath{X}}\\
\infty & \mbox{otherwise}.
\end{cases}
\]
The set $D(\Sigma):=\Sigma^{-1}(X)\subseteq X^{\Omega}$ will be called
its domain.
\item We define the \emph{constant part operator }$\constsymb:X^{\Omega}\to X\cup\{\infty\}$
as follows: If $\abs{\Omega}<\infty$, then, for $\xi\in X^{\Omega}$,
we define $\constmap\xi:=\abs{\Omega}^{-1}\sum_{\omega\in\Omega}\xi_{\omega}$.
If $\abs{\Omega}\nless\infty$, then we define the map $\constmap$
as follows: Let $\xi\in X^{\Omega}$. If there exists some $x\in X$
such that, for any $\varepsilon>0$, there exists some $F\in\finite{\Omega}$,
so that $\sup_{\omega\in\Omega\backslash F}\|\xi_{\omega}-x\|<\varepsilon$,
we define $\constmap\xi:=x$. If there exists no such $x$, we define
$\constmap\xi:=\infty$. The set $D(\constsymb):=\constmap^{-1}(X)\subseteq X^{\Omega}$
will be called its domain. 
\end{enumerate}
\end{defn}
{}
																\IfStrEqCase{\docclass}{{amsart}{\vfill\pagebreak}}
	\begin{defn}
		\label{def:sets-for-applications}Let $X$ be a Banach space. For
		some index set $\Omega$, let $\{C_{\omega}\}_{\omega\in\Omega}$
		a collection of cones in $X$. Let $Y\subseteq X^{\Omega}$ be a subspace.
		We define the following sets 
			\begin{enumerate}
				\item $(\directsum C)(Y):=\set{\xi\in Y}{\forall\omega\in\Omega,\ \xi_{\omega}\in C_{\omega}}$
				\item $\zerosummables(Y):=\set{\xi\in Y\cap D(\Sigma)}{\Sigma\xi=0}$
				\item $\onesummables(Y):=\set{\xi\in Y\cap D(\Sigma)}{\norm{\Sigma\xi}\leq1}$
				\item $\allconstants(Y):=\set{\xi\in Y\cap D(\constmap)}{\chi_{\Omega}\otimes\constmap\xi\in Y,\:\xi=\chi_{\Omega}\otimes\constmap\xi}$
				\item $\oneconstants(Y):=\allconstants(Y)\cap\set{\xi\in Y\cap D(\constmap)}{\norm{\constmap\xi}\leq1}$.
			\end{enumerate}
		It will often happen that we refer to a collection of sets of the above forms that all occur in a single space $Y$. 
		For the sake of readability, we will suppress repeated mention of $Y$ by introducing the following abbreviation when referring to such a collection. Explicitly, by the phrase 
		\[\textup{``}\directsum C,\ \zerosummables,\ \onesummables,\ \allconstants,\ \oneconstants\textup{ in $Y$}\textup{ ''},\] 
		we will mean
		\[\textup{``}(\directsum C)(Y),\ \ \zerosummables(Y),\ \ \onesummables(Y),\ \ \allconstants(Y),\ \ \oneconstants(Y)\textup{''}.\] 
	\end{defn}

To be clear as to our notation, if $(Y,Z)$ with $Y\subseteq X^{\Omega}$
and $Z\subseteq X'^{\Omega}$ is a dual pair with respect to $\Duality{\cdot}{\cdot}$,
we explicitly differentiate the meaning of ``$\directsum\bottompolar C$
in $Z$'', i.e., $(\directsum\bottompolar C)(Z)=\set{\eta\in Z}{\forall\omega\in\Omega,\ \eta_{\omega}\in\bottompolar{C_{\omega}}}$
-- the direct sum of the collection of one-sided polars $\curly{\bottompolar{C_{\omega}}}_{\omega\in\Omega}$,
from ``$\bottompolar{(\directsum C)}$'', i.e., $\bottompolar{((\directsum C)(Y))}=\set{\eta\in Z}{\forall\xi\in(\directsum C)(Y),\ \Duality{\xi}{\eta}\leq1}$
-- the one-sided polar of the direct sum $(\directsum C)(Y)$.
\begin{lem}
\label{lem:Polars-of-sums-and-constants}Let $X$ be a real Banach
space and $\{C_{\omega}\}_{\omega\in\Omega}$ a collection of closed
cones in $X$. 
\begin{enumerate}
\item In the dual pair $\parenth{\conv\omegadomain X,\l 1\omegadomain{X'}}$
the one-sided polars of the sets $\directsum C$, $\allconstants$,
$\oneconstants$, and $\closedball{\conv\omegadomain X}$ in $\conv\omegadomain X$
respectively equal the sets $\directsum\bottompolar C$, $\zerosummables$,
$\onesummables$ and $\closedball{\l 1\omegadomain{X'}}$ in $\l 1\omegadomain{X'}$.
\item In the dual pair $\parenth{\l 1\omegadomain X,\l{\infty}\omegadomain{X'}}$
the one-sided polars of the sets $\directsum C$, $\zerosummables$,
$\onesummables$ and $\closedball{\l 1\omegadomain X}$ in $\l 1\omegadomain X$
respectively equal the sets $\directsum\bottompolar C$, $\allconstants$,
$\oneconstants$ and $\closedball{\l{\infty}\omegadomain{X'}}$ in
$\l{\infty}\omegadomain{X'}$.
\item In the dual pair $\parenth{\l p\omegadomain X,\l q\omegadomain{X'}}$,
for $1\leq p,q\leq\infty$ with $p^{-1}+q^{-1}=1$ and $\abs{\Omega}<\infty$,
the one-sided polars of the sets $\directsum C$, $\allconstants$,
$\oneconstants$, $\onesummables$, $\zerosummables$ and $\closedball{\l p\omegadomain X}$
in $\l p\omegadomain X$ respectively equal the sets $\directsum\bottompolar C$,
$\zerosummables$, $\onesummables$, $\oneconstants$, $\allconstants$
and $\closedball{\l q\omegadomain{X'}}$ in $\l q\omegadomain{X'}$.
\end{enumerate}
\end{lem}
\begin{proof}
We prove (1):

It is clear that $\directsum\bottompolar C\subseteq\bottompolar{\parenth{\directsum C}}$.
We prove the reverse inclusion. Let $\eta\in\bottompolar{\parenth{\directsum C}}\subseteq\l 1\omegadomain{X'}$,
but suppose that $\eta\notin\directsum\bottompolar C$. By The Separation
Theorem \cite[Corollary IV.3.10]{Conway}, there exists some $\xi\in\conv\omegadomain X$
and $\alpha\in\R$ such that $\Duality{\xi}{\eta}>\alpha>\Duality{\xi}{\rho}$
for all $\rho\in\directsum\bottompolar C\subseteq\l 1\omegadomain{X'}$.
Since $0\in\directsum\bottompolar C$, we have $\alpha>0$. For every
$\omega\in\Omega$, $\lambda\geq0$ and $\phi\in\bottompolar{C_{\omega}}\subseteq X'$
we have $\delta_{\omega}\otimes(\lambda\phi)\in\directsum\bottompolar C$,
and therefore $\alpha>\Duality{\xi}{\delta_{\omega}\otimes(\lambda\phi)}=\lambda\duality{\xi_{\omega}}{\phi}$
implies $\duality{\xi_{\omega}}{\phi}\leq0\leq1$ for all $\omega\in\Omega$.
I.e., $\xi_{\omega}\in\bibottompolar{\parenth{C_{\omega}}}=C_{\omega}$
for all $\omega\in\Omega$, so that $\xi\in\directsum C\subseteq\conv\omegadomain X$.
For every $\lambda\geq0$, $\lambda\xi\in\directsum C$, so that $\Duality{\lambda\xi}{\eta}\leq1$
implies $\Duality{\xi}{\eta}\leq0$, yielding the absurdity $0<\alpha<\Duality{\xi}{\eta}\leq0$.
We conclude that $\eta\in\directsum\bottompolar C$, and hence that
$\directsum\bottompolar C=\bottompolar{\parenth{\directsum C}}$.

It is clear that $\zerosummables\subseteq\bottompolar{\allconstants}$.
We prove the reverse inclusion. Let $\eta\in\bottompolar{\allconstants}$,
but suppose $\eta\notin\zerosummables$. Then $\Sigma\eta\neq0$ implies
that there exists some $x\in X$ such that $\duality x{\Sigma\eta}>1$.
Since $\chi_{\Omega}\otimes x\in\allconstants\subseteq\conv\omegadomain X$,
we then have $\Duality{\chi_{\Omega}\otimes x}{\eta}=\duality x{\Sigma\eta}>1$,
so that $\eta\notin\bottompolar{\allconstants}$, contradicting our
assumption that $\eta\in\bottompolar{\allconstants}$. We conclude
$\bottompolar{\allconstants}\subseteq\zerosummables$. 

It is clear that $\onesummables\subseteq\bottompolar{\oneconstants}$.
We prove the reverse inclusion. Let $\eta\in\bottompolar{\oneconstants}$,
but suppose that $\eta\notin\onesummables$. Then there exists some
$x\in\closedball X$ such that $\duality x{\Sigma\eta}>1$. As before
$\Duality{\chi_{\Omega}\otimes x}{\eta}=\duality x{\Sigma\eta}>1$,
while $\chi_{\Omega}\otimes x\in\oneconstants\subseteq\conv\omegadomain X$,
contradicting our assumption that $\eta\in\bottompolar{\oneconstants}$.
We conclude $\bottompolar{\oneconstants}\subseteq\onesummables$.

Since $\l 1\omegadomain{X'}$ is the dual of $\conv\omegadomain X$,
it follows that $\closedball{\l 1\omegadomain{X'}}=\bottompolar{\closedball{\conv\omegadomain X}}$. 

We prove (2):

That $\directsum\bottompolar C=\bottompolar{\parenth{\directsum C}}$
follows as in (1). 

It is clear that $\allconstants\subseteq\bottompolar{\zerosummables}$.
We prove the reverse inclusion. Let $\eta\in\bottompolar{\zerosummables}$,
but suppose $\eta\notin\allconstants$. Then there exist $\omega_{0},\omega_{1}\in\Omega$
such that $\eta_{\omega_{0}}-\eta_{\omega_{1}}\neq0$. Let $x\in X$
be such that $\duality x{\eta_{\omega_{0}}-\eta_{\omega_{1}}}>1$.
Then $\delta_{\omega_{0}}\otimes x-\delta_{\omega_{1}}\otimes x\in\zerosummables$,
and $1\geq\Duality{\delta_{\omega_{0}}\otimes x-\delta_{\omega_{1}}\otimes x}{\eta}=\duality x{\eta_{\omega_{0}}-\eta_{\omega_{1}}}>1$,
which is absurd. We conclude that $\bottompolar{\zerosummables}\subseteq\allconstants$,
and hence $\bottompolar{\zerosummables}=\allconstants$.

It is clear that $\oneconstants\subseteq\bottompolar{\onesummables}$.
We prove the reverse inclusion. Let $\eta\in\bottompolar{\onesummables}$,
but suppose $\eta\notin\oneconstants$. Since $\zerosummables\subseteq\onesummables$,
by Lemma \ref{lem:elementary-polar-results}, $\bottompolar{\onesummables}\subseteq\bottompolar{\zerosummables}=\allconstants$.
Therefore $\eta\in\allconstants$, but since $\eta\notin\oneconstants$,
we have $\norm{\constmap\eta}>1$. Let $x\in\closedball X$ be such
that $\duality x{\constmap\eta}>1$, then, for any $\omega\in\Omega$,
we have $\delta_{\omega}\otimes x\in\onesummables$ and $1\geq\Duality{\delta_{\omega}\otimes x}{\eta}=\duality x{\constmap\eta}>1$,
which is absurd. We conclude that $\oneconstants\subseteq\bottompolar{\onesummables}$,
and hence $\oneconstants=\bottompolar{\onesummables}$.

Since $\l{\infty}\omegadomain{X'}$ is the dual $\l 1\omegadomain X$,
it follows that $\closedball{\l{\infty}\omegadomain{X'}}=\bottompolar{\closedball{\l 1\omegadomain X}}$. 

The result (3) follows as in (1) and (2). \end{proof}
\begin{rem}
If $\abs{\Omega}\nless\infty$ and $1<p<\infty$, then the canonical
summation operator $\Sigma$ on $D(\Sigma)\cap\l p\omegadomain{\R}$
is an unbounded non-closable operator (and hence also for $\l p\omegadomain X$
for any Banach space $X$). To see this, consider the sequence $\curly{\xi^{(n)}}_{n\in\N}\subseteq\l p\parenth{\N,\R}$,
defined by $\xi_{j}^{(n)}:=2^{-m-n}$ if $j\in\{(m-1)2^{n}+1,(m-1)2^{n}+2,\ldots,m2^{n}-1,m2^{n}\}$
for all $j,m,n\in\N$. Then $\Sigma\xi^{(n)}=\sum_{j=1}^{\infty}\xi_{j}^{(n)}=2^{n}\sum_{m=1}^{\infty}2^{-m-n}=1$
for all $n\in\N$, while 
\begin{eqnarray*}
\norm{\xi^{(n)}}_{p}^{p} & = & \sum_{j=1}^{\infty}\parenth{\xi_{j}^{(n)}}^{p}\\
 & = & 2^{-np+n}\sum_{m=1}^{\infty}2^{-mp}\\
 & = & 2^{-np+n}\frac{2^{-p}}{1-2^{-p}}
\end{eqnarray*}
implies $\xi^{(n)}\to0$ as $n\to\infty$. Using this observation,
one can show that the norm-closure of $\zerosummables$, and hence
also the norm-closure of $\onesummables$, is all of $\l p\parenth{\N,\R}$.
Therefore neither $\zerosummables$ nor $\onesummables$ is norm-closed,
and hence, not weakly closed. 

Our General Duality Theorems therefore do not apply when substituting
$\zerosummables$ for $C$ or $D$, or $\onesummables$ for $B_{1}$
or $B_{2}$ as in the hypotheses of Theorems \ref{thm:Normality-Duality}
and \ref{thm:Additivity-Duality} for the dual pair $\parenth{\l p\omegadomain X,\l q\omegadomain{X'}}$
in the case where $\abs{\Omega}\nless\infty$ and $1<p<\infty$ with
$p^{-1}+q^{-1}=1$.
\end{rem}
{}

Having computed the one-sided polars of the required sets in Lemma
\ref{lem:elementary-normality-polars}, it is now a routine matter
to apply our General Duality Theorems (Theorems \ref{thm:Normality-Duality}
and \ref{thm:Additivity-Duality}) to establish Theorems \ref{thm:general-banach-space-normality-duality}
and \ref{thm:general-banach-space-additivity-duality} below. Simple
verifications (left to the reader) will establish the reformulations
of Theorems \ref{thm:general-banach-space-normality-duality} and
\ref{thm:general-banach-space-additivity-duality} given in Corollaries
\ref{cor:general-banach-space-normality-duality-interpreted} and
\ref{cor:general-banach-space-additivity-duality-interpreted}. These
corollaries can be seen to generalize Theorems \ref{thm:classical-normality-duality}
and \ref{thm:classical-additivity-duality} from the introduction. %
																\IfStrEqCase{\docclass}{{amsart}{\vfill\pagebreak}}%
\begin{thm}
\label{thm:general-banach-space-normality-duality}Let $\alpha\geq1$.
Let $X$ be a real Banach space and $\{C_{\omega}\}_{\omega\in\Omega}$
a collection of closed cones in $X$. 
\begin{enumerate}
\item The space $\conv\omegadomain X$ is normal with respect to $\parenth{\directsum C,\allconstants,\alpha\oneconstants,\closedball{\conv\omegadomain X}}$
in $\conv\omegadomain X$ if and only if the space $\l 1\omegadomain{X'}$
is conormal with respect to $\parenth{\directsum\bottompolar C,\zerosummables,\onesummables,\alpha\closedball{\l 1\omegadomain{X'}}}$
in $\l 1\omegadomain{X'}$.
\item The space $\l 1\omegadomain X$ is conormal with respect to $\parenth{\directsum C,\zerosummables,\onesummables,\beta\closedball{\l 1\omegadomain X}}$
in $\l 1\omegadomain X$ for every $\beta>\alpha$ if and only if
the space $\l{\infty}\omegadomain{X'}$ is normal with respect to
$\parenth{\directsum\bottompolar C,\allconstants,\alpha\oneconstants,\closedball{\l{\infty}\omegadomain{X'}}}$
in $\l{\infty}\omegadomain{X'}$.
\end{enumerate}
\noindent If, in addition, $\abs{\Omega}<\infty$ and $1\leq p,q\leq\infty$,
with $p^{-1}+q^{-1}=1$, then:
\begin{enumerate}
\item [\listitemhack{3}]The space $\l p\omegadomain X$ is normal with respect to $\parenth{\directsum C,\allconstants,\alpha\oneconstants,\closedball{\l p\omegadomain X}}$
in $\l p\omegadomain X$ if and only if the space $\l q\omegadomain{X'}$
is conormal with respect to $\parenth{\directsum\bottompolar C,\zerosummables,\onesummables,\alpha\closedball{\l q\omegadomain{X'}}}$
in $\l q\omegadomain{X'}$.
\item [\listitemhack{4}]The space $\l p\omegadomain X$ is conormal with respect to
$\parenth{\directsum C,\zerosummables,\onesummables,\beta\closedball{\l p\omegadomain X}}$
in $\l p\omegadomain X$ for every $\beta>\alpha$ if and only if
the space $\l q\omegadomain{X'}$ is normal with respect to $\parenth{\directsum\bottompolar C,\allconstants,\alpha\oneconstants,\closedball{\l q\omegadomain{X'}}}$
in $\l q\omegadomain{X'}$.
\end{enumerate}
\end{thm}
\begin{proof}
We prove (1). By Lemma \ref{lem:Polars-of-sums-and-constants}, the
one-sided polars of the sets $\allconstants$, $\oneconstants$, $\directsum C$,
and $\closedball{\conv\omegadomain X}$ in $\conv\omegadomain X$
respectively are $\zerosummables$, $\onesummables$, $\directsum\bottompolar C$
and $\closedball{\l 1\omegadomain{X'}}$ in $\l 1\omegadomain{X'}$.
Since $\alpha\closedball{\l 1\omegadomain{X'}}$ is $\sigma\parenth{\l 1\omegadomain{X'},\conv\omegadomain X}$-compact
(by The Banach-Alaoglu Theorem \cite[Theorem V.3.1]{Conway}) and
the sets $\directsum\bottompolar C$ and $\zerosummables$ are 
		\IfStrEqCase{\docclass}{%
			{amsart}{$\sigma\parenth{\l 1\omegadomain{X'},\conv\omegadomain X}$-closed, }%
			{elsart}{$\sigma\parenth{\l 1\omegadomain{X'},\\\conv\omegadomain X}$-closed, }%
		}
$\alpha\closedball{\l 1\omegadomain{X'}}\cap\directsum\bottompolar C+\zerosummables$
is $\sigma\parenth{\l 1\omegadomain{X'},\conv\omegadomain X}$-closed.
The result now follows from Theorem \ref{thm:Normality-Duality}(1)(b). 

We prove (2). By Lemma \ref{lem:Polars-of-sums-and-constants}, the
one-sided polars of the sets $\zerosummables$, $\onesummables$,
$\directsum C$ and $\closedball{\l 1\omegadomain X}$ in $\l 1\omegadomain X$
respectively are $\allconstants$, $\oneconstants$, $\directsum\bottompolar C$
and $\closedball{\l{\infty}\omegadomain{X'}}$ in $\l{\infty}\omegadomain{X'}$.
By Theorem \ref{thm:Normality-Duality}(2)(a), if $\l 1\omegadomain X$
is conormal with respect to $\parenth{\directsum C,\zerosummables,\onesummables,\beta\closedball{\l 1\omegadomain X}}$
in $\l 1\omegadomain X$ for all $\beta>\alpha$, then the space $\l{\infty}\omegadomain{X'}$
is normal with respect to $\parenth{\directsum\bottompolar C,\allconstants,\alpha\oneconstants,\closedball{\l{\infty}\omegadomain{X'}}}$
in $\l{\infty}\omegadomain{X'}$.

		\IfStrEqCase{\docclass}{%
			{amsart}{%
				Conversely, let $\l{\infty}\omegadomain{X'}$ be normal with respect
				to 
				$\parenth{\directsum\bottompolar C,\allconstants,\alpha\oneconstants,\closedball{\l{\infty}\omegadomain{X'}}}$ 
				in $\l{\infty}\omegadomain{X'}$. Invoking Lemma \ref{lem:basic-cs-results},
				it can be seen that $\alpha\closedball{\l 1\omegadomain X}\cap\directsum C+\zerosummables$
				is $\sigma\parenth{\l 1\omegadomain X,\l{\infty}\omegadomain{X'}}$-cs-closed.
				Furthermore, the $\sigma\parenth{\l 1\omegadomain X,\l{\infty}\omegadomain{X'}}$-cs-compact
				set $\closedball{\l 1\omegadomain X}$ is contained in $\onesummables$.
				By Lemma \ref{lem:elementary-normality-polars}, 
				\[
				\onesummables\subseteq\closure[\sigma\parenth{\l 1\omegadomain X,\l{\infty}\omegadomain{X'}}]{\parenth{\alpha\closedball{\l 1\omegadomain X}\cap\directsum C+\zerosummables}},
				\]
				and since the norm-closure and $\sigma\parenth{\l 1\omegadomain X,\l{\infty}\omegadomain{X'}}$-closure
				of $\alpha\closedball{\l 1\omegadomain X}\cap\directsum C+\zerosummables$
				coincide (cf. \cite[Theorem V.1.4]{Conway}), it holds that, for every
				$r>0$ and $\xi\in\onesummables$, $\parenth{\xi-r\closedball{\l 1\omegadomain X}}\cap\parenth{\alpha\closedball{\l 1\omegadomain X}\cap\directsum C+\zerosummables}\neq\emptyset$.
				Finally, by Theorem \ref{thm:Normality-Duality}(2)(b), $\l 1\omegadomain X$
				is conormal with respect to $\parenth{\directsum C,\zerosummables,\onesummables,\beta\closedball{\l 1\omegadomain X}}$
				in $\l 1\omegadomain X$ for every $\beta>\alpha$.
			}%
			{elsart}{%
				Conversely, we let the space $\l{\infty}\omegadomain{X'}$ be normal with respect
				to 
				$\parenth{\directsum\bottompolar C,\allconstants,\\\alpha\oneconstants,\closedball{\l{\infty}\omegadomain{X'}}}$ 
				in $\l{\infty}\omegadomain{X'}$. Invoking Lemma \ref{lem:basic-cs-results},
				it can be seen that the set $\alpha\closedball{\l 1\omegadomain X}\cap\directsum C+\zerosummables$
				is $\sigma\parenth{\l 1\omegadomain X,\l{\infty}\omegadomain{X'}}$-cs-closed.
				Furthermore, the $\sigma\parenth{\l 1\omegadomain X,\l{\infty}\omegadomain{X'}}$-cs-compact
				set $\closedball{\l 1\omegadomain X}$ is contained in $\onesummables$.
				By Lemma~\ref{lem:elementary-normality-polars}, 
				\[
				\onesummables\subseteq\closure[\sigma\parenth{\l 1\omegadomain X,\l{\infty}\omegadomain{X'}}]{\parenth{\alpha\closedball{\l 1\omegadomain X}\cap\directsum C+\zerosummables}},
				\]
				and since the norm-closure and $\sigma\parenth{\l 1\omegadomain X,\l{\infty}\omegadomain{X'}}$-closure
				of $\alpha\closedball{\l 1\omegadomain X}\cap\directsum C+\zerosummables$
				coincide (cf. \cite[Theorem V.1.4]{Conway}), it holds that, for every
				$r>0$ and $\xi\in\onesummables$, $\parenth{\xi-r\closedball{\l 1\omegadomain X}}\cap\parenth{\alpha\closedball{\l 1\omegadomain X}\cap\directsum C+\zerosummables}\neq\emptyset$.
				Finally, by Theorem \ref{thm:Normality-Duality}(2)(b), $\l 1\omegadomain X$
				is conormal with respect to $\parenth{\directsum C,\zerosummables,\onesummables,\beta\closedball{\l 1\omegadomain X}}$
				in $\l 1\omegadomain X$ for every $\beta>\alpha$.
			}%
		}%

The assertions (3) and (4) follow similarly:

We prove (3). By Lemma \ref{lem:Polars-of-sums-and-constants}, the
one-sided polars of the sets $\allconstants$, $\oneconstants$, $\directsum C$,
and $\closedball{\l p\omegadomain X}$ in $\l p\omegadomain X$ respectively
are $\zerosummables$, $\onesummables$, $\directsum\bottompolar C$
and $\closedball{\l q\omegadomain{X'}}$ in $\l q\omegadomain X$.
As in (1), $\alpha\closedball{\l q\omegadomain{X'}}\cap\directsum\bottompolar C+\zerosummables$
is $\sigma\parenth{\l q\omegadomain{X'},\l p\omegadomain X}$-closed.
The result now follows from Theorem \ref{thm:Normality-Duality}(1)(b). 

We prove (4). By Theorem \ref{thm:Normality-Duality}(2)(a), if $\l p\omegadomain X$
is conormal with respect to $\parenth{\directsum C,\zerosummables,\onesummables,\beta\closedball{\l p\omegadomain X}}$
in $\l p\omegadomain X$ for all $\beta>\alpha$, then $\l q\omegadomain{X'}$
is normal with respect to $\parenth{\directsum\bottompolar C,\allconstants,\alpha\oneconstants,\closedball{\l q\omegadomain{X'}}}$
in $\l q\omegadomain{X'}$.

		\IfStrEqCase{\docclass}{%
			{amsart}{%
				Conversely, let $\l q\omegadomain{X'}$ be normal with respect to
				$\parenth{\directsum\bottompolar C,\allconstants,\alpha\oneconstants,\closedball{\l q\omegadomain{X'}}}$
				in $\l q\omegadomain{X'}$. Invoking Lemma \ref{lem:basic-cs-results},
				it can be seen that $\alpha\closedball{\l p\omegadomain X}\cap\directsum C+\zerosummables$
				is $\sigma\parenth{\l p\omegadomain X,\l q\omegadomain{X'}}$-cs-closed.
				Also, since $\abs{\Omega}$ is finite, by the H\"older-- and Minkowski
				inequalities, the $\norm{\cdot}_{1}$-- and $\norm{\cdot}_{p}$-norms
				on $\l p\omegadomain X$ are equivalent. By Lemma \ref{lem:basic-cs-results},
				$\closedball{\l 1\omegadomain X}$ is a $\sigma\parenth{\l p\omegadomain X,\l q\omegadomain{X'}}$-pre-cs-compact
				$\norm{\cdot}_{p}$-neighborhood of zero contained in $\onesummables$
				(being a subset of the cs-compact set $\gamma\closedball{\l p\omegadomain X}$
				for some $\gamma>0$). By Lemma \ref{lem:elementary-normality-polars},
				\[
				\onesummables\subseteq\closure[\sigma\parenth{\l p\omegadomain X,\l q\omegadomain{X'}}]{\parenth{\alpha\closedball{\l p\omegadomain X}\cap\directsum C+\zerosummables}},
				\]
				and since the norm-closure and $\sigma\parenth{\l p\omegadomain X,\l q\omegadomain{X'}}$-closure
				of $\alpha\closedball{\l p\omegadomain X}\cap\directsum C+\zerosummables$
				coincide (cf. \cite[Theorem V.1.4]{Conway}), it holds that, for every
				$r>0$ and $\xi\in\onesummables$, $\parenth{\xi-r\closedball{\l 1\omegadomain X}}\cap\parenth{\alpha\closedball{\l p\omegadomain X}\cap\directsum C+\zerosummables}\neq\emptyset$
				(since $\closedball{\l 1\omegadomain X}$ is a $\norm{\cdot}_{p}$-norm-neighborhood
				of zero). Finally, by Theorem \ref{thm:Normality-Duality}(2)(b),
				$\l p\omegadomain X$ is conormal with respect to $\parenth{\directsum C,\zerosummables,\onesummables,\beta\closedball{\l p\omegadomain X}}$
				in $\l p\omegadomain X$ for every $\beta>\alpha$.
			}%
			{elsart}{
				Conversely, we let the space $\l q\omegadomain{X'}$ be normal with respect to
				$\parenth{\directsum\bottompolar C,\allconstants,\\\alpha\oneconstants,\closedball{\l q\omegadomain{X'}}}$
				in $\l q\omegadomain{X'}$. Invoking Lemma \ref{lem:basic-cs-results},
				it can be seen that the set $\alpha\closedball{\l p\omegadomain X}\cap\directsum C+\zerosummables$
				is $\sigma\parenth{\l p\omegadomain X,\l q\omegadomain{X'}}$-cs-closed.
				Also, since $\abs{\Omega}$ is finite, by the H\"older-- and Minkowski
				inequalities, the $\norm{\cdot}_{1}$-- and $\norm{\cdot}_{p}$-norms
				on $\l p\omegadomain X$ are equivalent. By Lemma \ref{lem:basic-cs-results},
				$\closedball{\l 1\omegadomain X}$ is a $\sigma\parenth{\l p\omegadomain X,\l q\omegadomain{X'}}$-pre-cs-compact
				$\norm{\cdot}_{p}$-neighborhood of zero contained in $\onesummables$
				(being a subset of the cs-compact set $\gamma\closedball{\l p\omegadomain X}$
				for some $\gamma>0$). By Lemma \ref{lem:elementary-normality-polars},
				\[
				\onesummables\subseteq\closure[\sigma\parenth{\l p\omegadomain X,\l q\omegadomain{X'}}]{\parenth{\alpha\closedball{\l p\omegadomain X}\cap\directsum C+\zerosummables}},
				\]
				and since the norm-closure and $\sigma\parenth{\l p\omegadomain X,\l q\omegadomain{X'}}$-closure
				of $\alpha\closedball{\l p\omegadomain X}\cap\directsum C+\zerosummables$
				coincide (cf. \cite[Theorem V.1.4]{Conway}), it holds that, for every
				$r>0$ and $\xi\in\onesummables$, $\parenth{\xi-r\closedball{\l 1\omegadomain X}}\cap\parenth{\alpha\closedball{\l p\omegadomain X}\cap\directsum C+\zerosummables}\neq\emptyset$
				(since $\closedball{\l 1\omegadomain X}$ is a $\norm{\cdot}_{p}$-norm-neighborhood
				of zero). Finally, by Theorem \ref{thm:Normality-Duality}(2)(b),
				$\l p\omegadomain X$ is conormal with respect to $\parenth{\directsum C,\zerosummables,\onesummables,\beta\closedball{\l p\omegadomain X}}$
				in $\l p\omegadomain X$ for every $\beta>\alpha$.			
			}%
		}

\end{proof}
With a straightforward calculation, which we omit, the above theorem can be reformulated into the following corollary.
\begin{cor}
\label{cor:general-banach-space-normality-duality-interpreted}Let
$X$ be a real Banach space and $\{C_{\omega}\}_{\omega\in\Omega}$
a collection of closed cones in $X$ and $\alpha\geq1$.
\begin{enumerate}
\item The following are equivalent:

\begin{enumerate}
\item If $\xi\in\conv\omegadomain X$ and $x\in\bigcap_{\omega\in\Omega}\parenth{\xi_{\omega}+C_{\omega}}$,
then $\norm x\leq\alpha\norm{\xi}_{\infty}$.
\item For every $\phi\in X'$, there exists an element $\eta\in\directsum\bottompolar C\subseteq\l 1\omegadomain{X'}$
such that $\phi=\Sigma\eta$ and $\norm{\eta}_{1}\leq\alpha\norm{\phi}$.
\end{enumerate}
\item The following are equivalent:

\begin{enumerate}
\item For any $x\in X$ and $\beta>\alpha$, there exists an element $\xi\in\directsum C\subseteq\l 1\omegadomain X$
such that $x=\Sigma\xi$ and $\norm{\xi}_{1}\leq\beta\norm x$.
\item If $\eta\in\l{\infty}\omegadomain{X'}$ and $\phi\in\bigcap_{\omega\in\Omega}\parenth{\eta+\bottompolar{C_{\omega}}}$,
then $\norm{\phi}\leq\alpha\norm{\eta}_{\infty}$.
\end{enumerate}
\end{enumerate}
\noindent If, in addition, $\abs{\Omega}<\infty$ and $1\leq p,q\leq\infty$,
with $p^{-1}+q^{-1}=1$, then:
\begin{enumerate}
\item [\listitemhack{3}]The following are equivalent:

\begin{enumerate}
\item If $\xi\in\l p\omegadomain X$ and $x\in\bigcap_{\omega\in\Omega}\parenth{\xi_{\omega}+C_{\omega}}$,
then $\norm x\leq\alpha\norm{\xi}_{p}$.
\item For every $\phi\in X'$, there exists an element $\eta\in\directsum\bottompolar C\subseteq\l q\omegadomain{X'}$
such that $\phi=\Sigma\eta$ and $\norm{\eta}_{q}\leq\alpha\norm{\phi}$.
\end{enumerate}
\item [\listitemhack{4}]The following are equivalent:

\begin{enumerate}
\item For every $x\in X$ and $\beta>\alpha$, there exists an element $\xi\in\directsum C\subseteq\l p\omegadomain X$
such that $x=\Sigma\xi$ and $\norm{\xi}_{p}\leq\beta\norm x$.
\item If $\eta\in\l q\omegadomain{X'}$ and $\phi\in\bigcap_{\omega\in\Omega}\parenth{\eta+\bottompolar{C_{\omega}}}$,
then $\norm{\phi}\leq\alpha\norm{\eta}_{q}$.
\end{enumerate}
\end{enumerate}
\end{cor}
A similar argument as was employed in Theorem \ref{thm:general-banach-space-normality-duality}
will also establish the following theorem.
																\IfStrEqCase{\docclass}{{amsart}{\vfill\pagebreak}}
\begin{thm}
\label{thm:general-banach-space-additivity-duality}Let $\alpha\geq1$.
Let $X$ be a real Banach space and $\{C_{\omega}\}_{\omega\in\Omega}$
a collection of closed cones in $X$. 
\begin{enumerate}
\item The space $\l 1\omegadomain X$ is additive with respect to $\parenth{\directsum C,\{0\},\alpha\closedball{\l 1\omegadomain X},\onesummables}$
in $\l 1\omegadomain X$ if and only if the space $\l{\infty}\omegadomain{X'}$
is coadditive with respect to $\parenth{\directsum\bottompolar C,\l{\infty}\omegadomain{X'},\closedball{\l{\infty}\omegadomain{X'}},\alpha\oneconstants}$
in $\l{\infty}\omegadomain{X'}$.
		\IfStrEqCase{\docclass}{%
			{amsart}{%
				\item The space $\conv\omegadomain X$ is coadditive with respect to $\parenth{\directsum C,\conv\omegadomain X,\closedball{\conv\omegadomain X},\beta\oneconstants}$
				in $\conv\omegadomain X$ for every $\beta>\alpha$ if and only if
				the space $\l 1\omegadomain{X'}$ is additive with respect to $\parenth{\directsum\bottompolar C,\{0\},\alpha\closedball{\l 1\omegadomain{X'}},\onesummables}$
				in $\l 1\omegadomain{X'}$.
			}%
			{elsart}{%
				\item The space $\conv\omegadomain X$ is coadditive with respect to $\parenth{\directsum C,\ \conv\omegadomain X,\ \closedball{\conv\omegadomain X},\\\beta\oneconstants}$
				in $\conv\omegadomain X$ for every $\beta>\alpha$ if and only if
				the space $\l 1\omegadomain{X'}$ is additive with respect to $\parenth{\directsum\bottompolar C,\{0\},\alpha\closedball{\l 1\omegadomain{X'}},\onesummables}$
				in $\l 1\omegadomain{X'}$.
			}%
		}%
\end{enumerate}
\noindent If, in addition, $\abs{\Omega}<\infty$ and $1\leq p,q\leq\infty$,
with $p^{-1}+q^{-1}=1$, then:
\begin{enumerate}
\item [\listitemhack{3}]	The space $\l p\omegadomain X$ is additive with respect to
$\parenth{\directsum C,\{0\},\alpha\closedball{\l p\omegadomain X},\onesummables}$
in $\l p\omegadomain X$ if and only if the space $\l q\omegadomain{X'}$
is coadditive with respect to $\parenth{\directsum\bottompolar C,\l q\omegadomain{X'},\closedball{\l q\omegadomain{X'}},\alpha\oneconstants}$
in $\l q\omegadomain{X'}$.
\item [\listitemhack{4}]The space $\l p\omegadomain X$ is coadditive with respect to
$\parenth{\directsum C,\l p\omegadomain X,\closedball{\l p\omegadomain X},\\ \beta\oneconstants}$
in $\l p\omegadomain X$ for all $\beta>\alpha$ if and only if the
space $\l q\omegadomain{X'}$ is additive with respect to $\parenth{\directsum\bottompolar C,\{0\},\alpha\closedball{\l q\omegadomain{X'}},\onesummables}$
in $\l q\omegadomain{X'}$.
\end{enumerate}
\end{thm}
\begin{proof}
We prove (1). By Lemma \ref{lem:Polars-of-sums-and-constants}, the
one-sided polars of the sets $\onesummables$, $\directsum C$ and
$\closedball{\l 1\omegadomain{X'}}$ in $\l 1\omegadomain X$ respectively
equal $\oneconstants$, $\directsum\bottompolar C$, and $\closedball{\l{\infty}\omegadomain X}$
in $\l{\infty}\omegadomain{X'}$. We notice that $\alpha\oneconstants$
in $\l{\infty}\omegadomain{X'}$ is $\sigma\parenth{\l{\infty}\omegadomain{X'},\l 1\omegadomain X}$-compact,
since it is $\sigma\parenth{\l{\infty}\omegadomain{X'},\l 1\omegadomain X}$-closed
in $\closedball{\l{\infty}\omegadomain{X'}}$, and hence, the set
$\alpha\oneconstants+\directsum\bottompolar C$ is $\sigma\parenth{\l{\infty}\omegadomain{X'},\l 1\omegadomain X}$-closed.
By Theorem \ref{thm:Additivity-Duality}(1)(b) the result follows.

We prove (2). By Lemma \ref{lem:Polars-of-sums-and-constants}, the
one-sided polars of the sets $\oneconstants$, $\directsum C$, and
$\closedball{\conv\omegadomain X}$ in $\conv\omegadomain X$ respectively
are $\onesummables$, $\directsum\bottompolar C$ and $\closedball{\l 1\omegadomain{X'}}$
in $\l 1\omegadomain{X'}$. From Theorem \ref{thm:Additivity-Duality}(2)(a)
we immediately conclude that, if $\conv\omegadomain X$ is coadditive
with respect to $\parenth{\directsum C,\conv\omegadomain X,\closedball{\conv\omegadomain X},\beta\oneconstants}$
in $\conv\omegadomain X$ for every $\beta>\alpha$, then $\l 1\omegadomain{X'}$
is additive with respect to $\parenth{\directsum\bottompolar C,\{0\},\alpha\closedball{\l 1\omegadomain{X'}},\onesummables}$
in $\l 1\omegadomain{X'}$.

		\IfStrEqCase{\docclass}{%
			{amsart}{%
				Conversely, let $\l 1\omegadomain{X'}$ be additive with respect to
				$\parenth{\directsum\bottompolar C,\{0\},\alpha\closedball{\l 1\omegadomain{X'}},\onesummables}$
				in $\l 1\omegadomain{X'}$. Invoking Lemma \ref{lem:basic-cs-results},
				it can be seen that the set $\alpha\oneconstants+\directsum C$ is
				$\sigma\parenth{\conv\omegadomain X,\l 1\omegadomain{X'}}$-cs-closed
				($\alpha\oneconstants+\directsum C$ is $\sigma\parenth{\conv\omegadomain X,\l 1\omegadomain{X'}}$-closed
				and convex). By Lemma \ref{lem:elementary-normality-polars}, 
				\[
				\closedball{\conv\omegadomain X}\subseteq\closure[\sigma\parenth{\conv\omegadomain X,\l 1\omegadomain{X'}}]{\parenth{\alpha\oneconstants+\directsum C}}\cap\conv\omegadomain X=\closure[\sigma\parenth{\conv\omegadomain X,\l 1\omegadomain{X'}}]{\parenth{\alpha\oneconstants+\directsum C}}.
				\]
				Since the norm-closure and $\sigma\parenth{\conv\omegadomain X,\l 1\omegadomain{X'}}$-closure
				of $\alpha\oneconstants+\directsum C$ coincide (cf. \cite[Theorem V.1.4]{Conway}),
				we have, for every $r>0$ and $b\in\closedball{\conv\omegadomain X}$,
				that $\parenth{b-r\closedball{\conv\omegadomain X}}\cap\parenth{\alpha\oneconstants+\directsum C}\neq\emptyset$.
				But $\closedball{\conv\omegadomain X}$ is a $\sigma\parenth{\conv\omegadomain X,\l 1\omegadomain{X'}}$-cs-compact
				set, so by Theorem \ref{thm:Additivity-Duality}(2)(b), $\conv\omegadomain X$
				is coadditive with respect to $\parenth{\directsum C,\conv\omegadomain X,\closedball{\conv\omegadomain X},\beta\oneconstants}$
				in $\conv\omegadomain X$ for every $\beta>\alpha$.
			}%
			{elsart}{%
				Conversely, we let the space $\l 1\omegadomain{X'}$ be additive with respect to
				$\parenth{\directsum\bottompolar C,\{0\},\\\alpha\closedball{\l 1\omegadomain{X'}},\onesummables}$
				in $\l 1\omegadomain{X'}$. Invoking Lemma \ref{lem:basic-cs-results},
				it can be seen that the set $\alpha\oneconstants+\directsum C$ is
				$\sigma\parenth{\conv\omegadomain X,\l 1\omegadomain{X'}}$-cs-closed
				(the set $\alpha\oneconstants+\directsum C$ is $\sigma\parenth{\conv\omegadomain X,\\\l 1\omegadomain{X'}}$-closed
				and convex). By Lemma \ref{lem:elementary-normality-polars}, 
				\[
				\closedball{\conv\omegadomain X}\subseteq\closure[\sigma\parenth{\conv\omegadomain X,\l 1\omegadomain{X'}}]{\parenth{\alpha\oneconstants+\directsum C}}\cap\conv\omegadomain X=\closure[\sigma\parenth{\conv\omegadomain X,\l 1\omegadomain{X'}}]{\parenth{\alpha\oneconstants+\directsum C}}.
				\]
				Since the norm-closure and $\sigma\parenth{\conv\omegadomain X,\l 1\omegadomain{X'}}$-closure
				of $\alpha\oneconstants+\directsum C$ coincide (cf. \cite[Theorem V.1.4]{Conway}),
				we have, for every $r>0$ and $b\in\closedball{\conv\omegadomain X}$,
				that $\parenth{b-r\closedball{\conv\omegadomain X}}\cap\parenth{\alpha\oneconstants+\directsum C}\neq\emptyset$.
				But $\closedball{\conv\omegadomain X}$ is a $\sigma\parenth{\conv\omegadomain X,\l 1\omegadomain{X'}}$-cs-compact
				set, so by Theorem \ref{thm:Additivity-Duality}(2)(b), $\conv\omegadomain X$
				is coadditive with respect to $\parenth{\directsum C,\conv\omegadomain X,\closedball{\conv\omegadomain X},\beta\oneconstants}$
				in $\conv\omegadomain X$ for every $\beta>\alpha$.
			}%
		}

The assertions (3) and (4) follow similarly:

We prove (3). By Lemma \ref{lem:Polars-of-sums-and-constants}, the
one-sided polars of the sets $\onesummables$, $\directsum C$ and
$\closedball{\l p\omegadomain X}$ in $\l p\omegadomain X$ respectively
equal $\oneconstants$, $\directsum\bottompolar C$, and $\closedball{\l q\omegadomain{X'}}$
in $\l q\omegadomain{X'}$. We notice $\alpha\oneconstants+\directsum\bottompolar C$
is $\sigma\parenth{\l q\omegadomain{X'},\l p\omegadomain X}$-closed.
By Theorem \ref{thm:Additivity-Duality}(1)(b) the result follows.

We prove (4). By Lemma \ref{lem:Polars-of-sums-and-constants}, the
one-sided polars of the sets $\oneconstants$, $\directsum C$, and
$\closedball{\l p\omegadomain X}$ in $\l p\omegadomain X$ respectively
are $\onesummables$, $\directsum\bottompolar C$ and $\closedball{\l q\omegadomain{X'}}$
in $\l q\omegadomain{X'}$. From Theorem \ref{thm:Additivity-Duality}(2)(a)
we immediately conclude that, if $\l p\omegadomain X$ is coadditive
with respect to $\parenth{\directsum C,\l p\omegadomain X,\closedball{\l p\omegadomain X},\beta\oneconstants}$
in $\l p\omegadomain X$ for every $\beta>\alpha$, then $\l q\omegadomain{X'}$
is additive with respect to $\parenth{\directsum\bottompolar C,\{0\},\alpha\closedball{\l q\omegadomain{X'}},\onesummables}$
in $\l q\omegadomain{X'}$.

Conversely, let $\l q\omegadomain{X'}$ be additive with respect to
		\IfStrEqCase{\docclass}{%
			{amsart}{
				$\parenth{\directsum\bottompolar C,\{0\},\alpha\closedball{\l q\omegadomain{X'}},\onesummables}$
			}%
			{elsart}{
				$\parenth{\directsum\bottompolar C,\{0\},\alpha\closedball{\l q\omegadomain{X'}},\\\onesummables}$
			}%
		}
in $\l q\omegadomain{X'}$. Invoking Lemma \ref{lem:basic-cs-results},
it can be is seen that the set $\alpha\oneconstants+\directsum C$
is $\sigma\parenth{\l p\omegadomain X,\l q\omegadomain{X'}}$-cs-closed.
By Lemma \ref{lem:elementary-normality-polars}, 
		\IfStrEqCase{\docclass}{%
			{amsart}{
				\[
				\closedball{\l p\omegadomain X}\subseteq\closure[\sigma\parenth{\l p\omegadomain X,\l q\omegadomain{X'}}]{\parenth{\alpha\oneconstants+\directsum C}}\cap\l p\omegadomain X=\closure[\sigma\parenth{\l p\omegadomain X,\l q\omegadomain{X'}}]{\parenth{\alpha\oneconstants+\directsum C}}.
				\]%
			}%
			{elsart}{
				\begin{eqnarray*}
				\closedball{\l p\omegadomain X} 	&\subseteq&			\closure[\sigma\parenth{\l p\omegadomain X,\l q\omegadomain{X'}}]{\parenth{\alpha\oneconstants+\directsum C}}\cap\l p\omegadomain X \\
													&=&				\closure[\sigma\parenth{\l p\omegadomain X,\l q\omegadomain{X'}}]{\parenth{\alpha\oneconstants+\directsum C}}.
				\end{eqnarray*}%
			}%
		}
Since the norm-closure and $\sigma\parenth{\l p\omegadomain X,\l q\omegadomain{X'}}$-closure
of $\alpha\oneconstants+\directsum C$ coincide (cf. \cite[Theorem V.1.4]{Conway}),
we have, for every $r>0$ and $b\in\closedball{\l p\omegadomain X}$,
that $\parenth{b-r\closedball{\l p\omegadomain X}}\cap\parenth{\alpha\oneconstants+\directsum C}\neq\emptyset$.
But $\closedball{\l p\omegadomain X}$ is a $\sigma\parenth{\l p\omegadomain X,\l q\omegadomain{X'}}$-cs-compact
set, so by Theorem \ref{thm:Additivity-Duality}(2)(b), $\l p\omegadomain X$
is coadditive with respect to $\parenth{\directsum C,\l p\omegadomain X,\closedball{\l p\omegadomain X},\beta\oneconstants}$
in $\l p\omegadomain X$ for every $\beta>\alpha$.
\end{proof}
Again, a straightforward calculation which we omit, allows the reformulation of the above theorem into the following corollary.
\begin{cor}
\label{cor:general-banach-space-additivity-duality-interpreted}Let
$X$ be a real Banach space and $\{C_{\omega}\}_{\omega\in\Omega}$
a collection of closed cones in $X$ and $\alpha\geq1$.
\begin{enumerate}
\item The following are equivalent:

\begin{enumerate}
\item If $\xi\in\directsum C\subseteq\l 1\omegadomain X$, then $\norm{\xi}_{1}\leq\alpha\norm{\Sigma\xi}$.
\item For every $\eta\in\l{\infty}\omegadomain{X'}$, there exists some
$\phi\in\bigcap_{\omega\in\Omega}\parenth{\eta_{\omega}-\bottompolar{C_{\omega}}}$
with $\norm{\phi}\leq\alpha\norm{\eta}_{\infty}$.
\end{enumerate}
\item The following are equivalent:

\begin{enumerate}
\item For every $\xi\in\conv\omegadomain X$ and $\beta>\alpha$, there
exists some $x\in\bigcap_{\omega\in\Omega}\parenth{\xi_{\omega}-C_{\omega}}$
with $\norm x\leq\beta\norm{\xi}_{\infty}$.
\item If $\eta\in\directsum\bottompolar C\subseteq\l 1\omegadomain{X'}$,
then $\norm{\eta}_{1}\leq\alpha\norm{\Sigma\eta}$.
\end{enumerate}
\end{enumerate}
\noindent If, in addition, $\abs{\Omega}<\infty$ and $1\leq p,q\leq\infty$,
with $p^{-1}+q^{-1}=1$, then:
\begin{enumerate}
\item [\listitemhack{3}]The following are equivalent:

\begin{enumerate}
\item If $\xi\in\directsum C\subseteq\l p\omegadomain X$, then $\norm{\xi}_{p}\leq\alpha\norm{\Sigma\xi}$.
\item For every $\eta\in\l q\omegadomain{X'}$, there exists some $\phi\in\bigcap_{\omega\in\Omega}\parenth{\eta_{\omega}-\bottompolar{C_{\omega}}}$
with $\norm{\phi}\leq\alpha\norm{\eta}_{q}$.
\end{enumerate}
\item [\listitemhack{4}]The following are equivalent:

\begin{enumerate}
\item For every $\xi\in\l p\omegadomain X$ and $\beta>\alpha$, there exists
some $x\in\bigcap_{\omega\in\Omega}\parenth{\xi_{\omega}-C_{\omega}}$
with $\norm x\leq\beta\norm{\xi}_{p}$.
\item If $\eta\in\directsum\bottompolar C\subseteq\l q\omegadomain{X'}$,
then $\norm{\eta}_{q}\leq\alpha\norm{\Sigma\eta}$.
\end{enumerate}
\end{enumerate}
\end{cor}

\section{\label{sec:cstar}Application: Geometric properties of cones in C{*}-algebras}

In this section we give an elementary application of geometric duality
theory to naturally occurring cones in C{*}-algebras. 

Throughout this section $A$ will be a C{*}-algebra and $A'$ its
dual. We will view both $A$ and $A'$ as vector spaces over $\R$,
and define the (real) bilinear map $\duality{\cdot}{\cdot}:A\times A'\to\R$
by $\duality a{\phi}:=\Realpart\phi(a)\ (a\in A,\ \phi\in A')$. By
The Separation Theorem \cite[Theorem 3.21]{Rudin}, $\duality{\cdot}{\cdot}$
becomes a duality as defined in Section \ref{sub:polar-calculus}.
As usual, we define the closed cone of positive elements in $A$ by
$A_{+}:=\set{a^{*}a\in A}{a\in A}$ and the closed cone of positive
functionals in $A'$ by $A_{+}':=\set{\phi\in A'}{\phi(a^{*}a)\geq0\:\forall a\in A}$.
We stress, since $A$ is a complex space, that $A_{+}'\neq-\bottompolar{A_{+}}=\set{\phi\in A'}{\Realpart\phi(a^{*}a)\geq0,\:\forall a\in A}$.
For $a,b\in A$ and $\phi,\varphi\in A'$, by $a\leq b$ and $\phi\leq\varphi$
we respectively mean $b\in a+A_{+}$ and $\varphi\in\phi+A_{+}'$.
We define the real subspaces of self-adjoint elements in $A$ by $\selfadjoint A:=\set{\phi\in A}{a=a^{*}}$
and of self-adjoint functionals on $A$ by $\selfadjoint A':=\set{\phi\in A'}{\phi(a^{*})=\overline{\phi(a)},\ \forall a\in A}$. 
\begin{lem}
\label{lem:cstar-cone-polars}Let $A$ be a C{*}-algebra and $A'$
its dual, and the duality $\duality{\cdot}{\cdot}:A\times A'\to\R$
as defined above. Then,
\begin{enumerate}
\item $\bottompolar{A_{+}}=-A_{+}'+i\selfadjoint A'$.
\item $\bottompolar{\parenth{iA_{+}}}=-iA_{+}'+\selfadjoint A'$
\item $\bottompolar{\parenth{A_{+}'}}=-A_{+}+i\selfadjoint A$.
\item $\bottompolar{\parenth{iA_{+}'}}=-iA_{+}+\selfadjoint A$.
\end{enumerate}
\end{lem}
\begin{proof}
We prove (1). For $\phi\in-A_{+}'$, $\varphi\in\selfadjoint A'$
and $a\in A$, since $\varphi(a^{*}a)\in\R$, it is clear that $\Realpart(\phi+i\varphi)(a^{*}a)=\phi(a^{*}a)\leq0$.
Hence $-A_{+}'+i\selfadjoint A'\subseteq\bottompolar{A_{+}}$. Conversely,
let $\phi\in\bottompolar{A_{+}}$. Defining $\phi^{*}\in A'$ by $\phi^{*}(a):=\overline{\phi(a^{*})}\ (a\in A)$,
and $\varphi,\psi\in A$ by $\varphi:=2^{-1}(\phi+\phi^{*})$ and
$\psi:=(2i)^{-1}(\phi-\phi^{*})$, so that $\varphi,\psi\in\selfadjoint A'$
and $\phi=\varphi+i\psi$. Hence, for all $a\in A$, $\varphi(a^{*}a)=\Realpart\phi(a^{*}a)=\duality{a^{*}a}{\phi}\leq0$,
so that $\varphi\in-A_{+}'$ and $\bottompolar{A_{+}}\subseteq-A_{+}'+i\selfadjoint A'$.
A similar argument will establish (2).

We prove (3). Let $a\in-A_{+}$ and $b\in\selfadjoint A$. Then for
$\phi\in A_{+}'$, $\duality{a+ib}{\phi}=\Realpart(\phi(a)+i\phi(b))=\phi(a)\leq0$,
so that $a+ib\in\bottompolar{A_{+}'}$. Conversely, let $a\in\bottompolar{A_{+}'}$.
We write $a=2^{-1}(a+a^{*})+i(2i)^{-1}(a-a^{*})$. Then, for all $\phi\in A_{+}'$,
we have $0\geq\duality a{\phi}=\Realpart\phi(a)=2^{-1}\phi(a+a^{*})$.
Therefore, $2^{-1}(a+a^{*})\in-A_{+}$ (by \cite[Proposition 2.6.2]{Dixmier}),
and hence $a=2^{-1}(a+a^{*})+i(2i)^{-1}(a-a^{*})\in-A_{+}+i\selfadjoint A$.
A similar argument will establish (4).
\end{proof}
The following result, originally due to Grothendieck (\cite{Grothendieck}
via \cite[Theorem~3.2.5]{Pedersen}), shows that the cones $\{A_{+}',-A_{+}',iA_{+}',-iA_{+}'\}$
generate the dual of a C{*}-algebra $A$. 
\begin{thm}
\label{thm:Grothendieck-Jordan}Let $A$ be a C{*}-algebra. If $\phi\in\selfadjoint A'$,
then there exist unique $\phi_{+},\phi_{-}\in A_{+}'$ with $\phi=\phi_{+}-\phi_{-}$
and $\norm{\phi}=\norm{\phi_{+}}+\norm{\phi_{-}}$.
\end{thm}
We can now apply our results from the previous sections to obtain
the following geometric properties of naturally occurring cones in
C{*}-algebras and their duals. The results presented here are stronger
than what is usually presented in the canon (e.g., \cite[1.6.9]{Dixmier}). 
\begin{thm}
    \label{thm:cstar-order-results}Let $A$ be a C{*}-algebra and $A'$
    its dual.
    \begin{enumerate}
    \item For $b_{1},b_{2},b_{3},b_{4}\in A$, if
    \begin{eqnarray*}
     &  & a\in\parenth{b_{1}+A_{+}+i\selfadjoint A}\cap\parenth{b_{2}-A_{+}+i\selfadjoint A}\\
     &  & \qquad\qquad\cap\parenth{b_{3}+iA_{+}+\selfadjoint A}\cap\parenth{b_{4}-iA_{+}+\selfadjoint A}
    \end{eqnarray*}
    then $\norm a\leq\max\curly{\norm{b_{1}},\norm{b_{2}}}+\max\curly{\norm{b_{3}},\norm{b_{4}}}$.
    \item For $b_{1},b_{2},b_{3},b_{4}\in A$, if
    \begin{eqnarray*}
     &  & a\in\parenth{b_{1}+A_{+}}\cap\parenth{b_{2}-A_{+}}\cap\parenth{b_{3}+iA_{+}}\cap\parenth{b_{4}-iA_{+}},
    \end{eqnarray*}
    then $\norm a\leq\max\curly{\norm{b_{1}},\norm{b_{2}}}+\max\curly{\norm{b_{3}},\norm{b_{4}}}$.
    \item For $a,b,c\in A$, if $a\leq b\leq c$, then $\norm b\leq2\max\curly{\norm a,\norm c}$.
    \item For $a,b,c\in\selfadjoint A$, if $a\leq b\leq c$, then $\norm b\leq\max\curly{\norm a,\norm c}$.
    \item For $\phi_{1},\phi_{2},\phi_{3},\phi_{4}\in A'$, if
    \begin{eqnarray*}
     &  & \varphi\in\parenth{\phi_{1}+A_{+}'+i\selfadjoint A'}\cap\parenth{\phi_{2}-A_{+}'+i\selfadjoint A'}\\
     &  & \qquad\qquad\cap\parenth{\phi_{3}+iA_{+}'+\selfadjoint A'}\cap\parenth{\phi_{4}-iA_{+}'+\selfadjoint A'}
    \end{eqnarray*}
    then $\norm{\varphi}\leq\sum_{j=1}^{4}\norm{\phi_{j}}$.
    \item For $\phi_{1},\phi_{2},\phi_{3},\phi_{4}\in A'$, if
    \begin{eqnarray*}
     &  & \varphi\in\parenth{\phi_{1}+A_{+}'}\cap\parenth{\phi_{2}-A_{+}'}\cap\parenth{\phi_{3}+iA_{+}'}\cap\parenth{\phi_{4}-iA_{+}'},
    \end{eqnarray*}
    then $\norm{\varphi}\leq\sum_{j=1}^{4}\norm{\phi_{j}}$.
    \item For $\rho,\phi,\psi\in A'$, if $\rho\leq\phi\leq\psi$, then
    $\norm{\phi}\leq2\parenth{\norm{\rho}+\norm{\psi}}$.
    \item For $\rho,\phi,\psi\in\selfadjoint A'$, if $\rho\leq\phi\leq\psi$,
    then $\norm{\phi}\leq\norm{\rho}+\norm{\psi}$.
    \end{enumerate}
\end{thm}
\begin{proof}
We prove (1) by applying Theorem \ref{thm:Normality-Duality}. We
define the Banach spaces $Y:=(A\oplus_{\infty}A)\oplus_{1}(A\oplus_{\infty}A)$
and $Z:=(A'\oplus_{1}A')\oplus_{\infty}(A'\oplus_{1}A')$, and define
the duality $\Duality{\cdot}{\cdot}:Y\times Z\to\R$ by 
\[
\Duality{(a_{1},a_{2},a_{3},a_{4})}{(\phi_{1},\phi_{2},\phi_{3},\phi_{4})}:=\sum_{j=1}^{4}\duality{a_{j}}{\phi_{j}}=\sum_{j=1}^{4}\Realpart\phi_{j}(a_{j}),
\]
where $(a_{1},a_{2},a_{3},a_{4})\in Y$ and $(\phi_{1},\phi_{2},\phi_{3},\phi_{4})\in Z$.
We define the sets
\begin{eqnarray*}
C & := & \set{(a_{0,0},a_{0,1},a_{1,0},a_{1,1})\in Y}{a_{j,k}\in(-1)^{k}(i)^{j}A_{+}+(i)^{1-j}\selfadjoint A},\\
\allconstants & := & \set{(a,a,a,a)\in Y}{a\in A},\\
\oneconstants & := & \set{(a,a,a,a)\in Y}{a\in A,\ \norm a\leq1},\\
E & := & \set{(\phi_{0,0},\phi_{0,1},\phi_{1,0},\phi_{1,1})\in Z}{\phi_{j,k}\in(-1)^{1-k}(i)^{j}A_{+}'},\\
\zerosummables & := & \set{(\phi_{1},\phi_{2},\phi_{3},\phi_{4})\in Z}{\sum_{j=1}^{4}\phi_{j}=0},\\
\onesummables & := & \set{(\phi_{1},\phi_{2},\phi_{3},\phi_{4})\in Z}{\norm{\sum_{j=1}^{4}\phi_{j}}\leq1}.
\end{eqnarray*}
An easy computation together with Lemma \ref{lem:cstar-cone-polars}
will show that $\bottompolar C=E$, $\bottompolar{\allconstants}=\zerosummables$,
$\bottompolar{\oneconstants}=\onesummables$ and $\bottompolar{\closedball Y}=\closedball Z$.
By Theorem \ref{thm:Grothendieck-Jordan}, for every $\zeta\in\onesummables$,
there exist $\phi_{1},\phi_{2},\phi_{3},\phi_{4}\in A_{+}'$ with
$\sum_{j=1}^{4}\zeta_{j}=\phi_{2}-\phi_{1}+i(\phi_{4}-\phi_{3})$
and 
\[
\max\curly{\norm{\phi_{1}}+\norm{\phi_{2}},\norm{\phi_{4}}+\norm{\phi_{3}}}\leq\norm{\sum_{j=1}^{4}\zeta_{j}}\leq1,
\]
i.e., $\Phi:=(\phi_{1},\phi_{2},\phi_{3},\phi_{4})\in\closedball Z\cap E$.
Therefore, $\zeta=\Phi+(\zeta-\Phi)\in\closedball Z\cap E+\zerosummables$
and we conclude that $\onesummables\subseteq\closedball Z\cap E+\zerosummables$,
i.e., $Z$ is conormal with respect to $(E,\zerosummables,\onesummables,\closedball Z)$.
Since $\closedball Z$ is $\sigma(Z,Y)$-compact (by The Banach-Alaoglu
Theorem \cite[Theorem V.3.1]{Conway}) and $\zerosummables$ is $\sigma(Z,Y)$-closed,
$\closedball Z\cap E+\zerosummables$ is $\sigma(Z,Y)$-closed. Hence,
by Theorem~\ref{thm:Normality-Duality}(1), $Y$ is normal with respect
to $(C,\allconstants,\oneconstants,\closedball Y)$, i.e., $(\closedball Y+C)\cap\allconstants\subseteq\oneconstants$.
Therefore, for $b_{1},b_{2},b_{3},b_{4}\in A$, if
\[
a\in\parenth{b_{1}+A_{+}+i\selfadjoint A}\cap\parenth{b_{2}-A_{+}+i\selfadjoint A}\cap\parenth{b_{3}+iA_{+}+\selfadjoint A}\cap\parenth{b_{4}-iA_{+}+\selfadjoint A},
\]
 then $\norm a\leq\max\curly{\norm{b_{1}},\norm{b_{2}}}+\max\curly{\norm{b_{3}},\norm{b_{4}}}$.

We prove (5). For every $a\in A$, there exist elements $a_{1},a_{2},a_{3},a_{4}\in A_{+}$
with $a=a_{1}-a_{2}+ia_{3}-ia_{4}$ and $\max_{j\in\{1,2,3,4\}}\norm{a_{j}}\leq\norm a$.
By Lemma \ref{lem:cstar-cone-polars} and Corollary~\ref{cor:general-banach-space-normality-duality-interpreted}(4),
for $\phi_{j,k}\in A'\ (j,k\in\{0,1\})$, if $\varphi\in\bigcap_{j,k\in\{0,1\}}\parenth{\phi_{j,k}+(-1)^{j}(i)^{k}A_{+}'+(i)^{1-k}\selfadjoint A'}$,
then $\norm{\varphi}\leq\sum_{j,k\in\{0,1\}}\norm{\phi_{j,k}}$.

The assertions (2) and (6) follow immediately from (1) and (5) respectively,
since $A_{+}\subseteq A_{+}+\selfadjoint A$ and $A_{+}'\subseteq A_{+}'+\selfadjoint{A'}$.

We prove (3). If $a,b,c\in A$ satisfy $a\leq b\leq c$, then
\begin{eqnarray*}
b & \in & \parenth{a+A_{+}}\cap\parenth{c-A_{+}}\cap\parenth{a+\selfadjoint A}\cap\parenth{c+\selfadjoint A}\\
 & \subseteq & \parenth{a+A_{+}+i\selfadjoint A}\cap\parenth{c-A_{+}+i\selfadjoint A}\cap\parenth{a+iA_{+}+\selfadjoint A}\cap\parenth{c-iA_{+}+\selfadjoint A}.
\end{eqnarray*}
Hence, by (1), we obtain $b\leq2\max\curly{\norm a,\norm c}$.

We prove (4). If $a,b,c\in\selfadjoint A$ satisfy $a\leq b\leq c$,
then 
\begin{eqnarray*}
b & \in & \parenth{a+A_{+}}\cap\parenth{c-A_{+}}\cap\parenth{0+\selfadjoint A}\cap\parenth{0+\selfadjoint A}\\
 & \subseteq & \parenth{a+A_{+}+i\selfadjoint A}\cap\parenth{c-A_{+}+i\selfadjoint A}\cap\parenth{0+iA_{+}+\selfadjoint A}\cap\parenth{0-iA_{+}+\selfadjoint A}.
\end{eqnarray*}
Hence, by (1), we obtain $b\leq\max\curly{\norm a,\norm c}$.

We prove (7). If $\rho,\phi,\psi\in A'$ satisfy $\rho\leq\phi\leq\psi$,
then 
\begin{eqnarray*}
\phi & \in & \parenth{\rho+A_{+}'}\cap\parenth{\psi-A_{+}'}\cap\parenth{\rho+\selfadjoint A'}\cap\parenth{\psi+\selfadjoint A'}\\
 & \subseteq & \parenth{\rho+A_{+}'+i\selfadjoint A'}\cap\parenth{\psi-A_{+}'+i\selfadjoint A'}\cap\parenth{\rho+iA_{+}'+\selfadjoint A'}\cap\parenth{\psi-iA_{+}'+\selfadjoint A'},
\end{eqnarray*}
so, by (5), $\norm{\phi}\leq2\parenth{\norm{\rho}+\norm{\psi}}.$

We prove (8). If $\rho,\phi,\psi\in\selfadjoint A'$ satisfy $\rho\leq\phi\leq\psi$,
then 
\begin{eqnarray*}
\phi & \in & \parenth{\rho+A_{+}'}\cap\parenth{\psi-A_{+}'}\cap\parenth{0+\selfadjoint A'}\cap\parenth{0+\selfadjoint A'}\\
 & \subseteq & \parenth{\rho+A_{+}'+i\selfadjoint A'}\cap\parenth{\psi-A_{+}'+i\selfadjoint A'}\cap\parenth{0+iA_{+}'+\selfadjoint A'}\cap\parenth{0-iA_{+}'+\selfadjoint A'},
\end{eqnarray*}
so, by (5), $\norm{\phi}\leq\norm{\rho}+\norm{\psi}.$\end{proof}
\begin{rem}
Some of the above results are known: The earliest references to (4) and
(8) known to the author is \cite[Examples 1.1.7 and 1.2.5]{BattyRobinson}.
Particularly, (4) can be established through an elementary application
of Grothendieck's 1957 result Theorem \ref{thm:Grothendieck-Jordan},
and Grosberg and Krein's 1939 result Theorem \ref{thm:classical-normality-duality}(1)
from the introduction, so is (at least in theory) quite old. No references
to (1)--(3) and (5)--(7) are known to the author.
\end{rem}
%
		\bibliographystyle{amsplain}
		\bibliography{bibliography}

\end{document}